\documentclass[12pt]{article}
\usepackage{amsfonts}

\usepackage{graphicx}
\usepackage{amsmath}

\newtheorem{theorem}{Theorem}[section]

\newtheorem{corollary}[theorem]{Corollary}

\newtheorem{definition}[theorem]{Definition}

\newtheorem{lemma}[theorem]{Lemma}

\newtheorem{proposition}[theorem]{Proposition}

\numberwithin{equation}{section}

\newenvironment{proof}[1][Proof]{\textbf{#1.} }{\ \rule{0.5em}{0.5em}}

\begin{document}

\title{A Large Deviation Principle for Martingales\\
over Brownian Filtration}
\author{\textsc{By Z. Qian \ \ and \ \ C. Xu} \\
{\small {Mathematical Institute, University of Oxford}}}
\maketitle

\leftskip1truecm \rightskip1truecm

\noindent {\textbf{Abstract.}} In this article we establish a large
deviation principle for the family $\{\nu _{\varepsilon }:\varepsilon \in
(0,1)\}$ of \ distributions of the scaled stochastic processes $\{P_{-\log 
\sqrt{\varepsilon }}Z_{t}\}_{t\leq 1}$, where $(Z_{t})_{t\in \lbrack 0,1]}$
is a square-integrable martingale over Brownian filtration and $%
(P_{t})_{t\geq 0}$ is the Ornstein-Uhlenbeck semigroup. The rate function is
identified as well in terms of the Wiener-It\^{o} chaos decomposition of the
terminal value $Z_{1}$. The result is established by developing a continuity
theorem for large deviations, together with two essential tools, the
hypercontractivity of the Ornstein-Uhlenbeck semigroup and Lyons' continuity
theorem for solutions of Stratonovich type stochastic differential equations.

\leftskip0truecm \rightskip0truecm

\vskip0.5truecm


\vskip1truecm


\newpage

\tableofcontents

\section{Introduction}

In 1938, H. Cram\`{e}r \cite{cramer} published a result on the probability
of large deviations in the law of large numbers for sums of independent real
random variables, and some years later, H. Chernoff \cite{chernoff1952}
proved a general Caram\`{e}r's theorem. Cram\`{e}r's paper marked the
beginning of the study of large deviations of distributions towards their
limiting law. Schilder \cite{Schilder}, mainly developed from his Ph. D.
thesis, proved a beautiful theorem for large deviations of Brownian motion,
and thus opened the study of large deviations for probability measures on
spaces of continuous paths. Schilder's analysis in \cite{Schilder} proved
fundamental in dealing with functional integrations over function spaces. It
took some years, however, in particular in the hands of Azencott \cite
{azencott1}, Donsker-Varadhan \cite{Var66}, \cite{Var84}, Freidlin-Ventcel 
\cite{Fre-Wen}, \cite{ventcel-f1}, Stroock \cite{stroock1}, Deuschel-Stroock 
\cite{Deu-Str}, Dembo-Zeitouni \cite{LD techniques by Dembo and Zeitouni},
Dupuis-Ellis \cite{dupuis-ellis} and etc. to turn the results of large
deviations and the techniques developed to prove them into what nowadays we
may call the theory of large deviations.

Large deviation principles have been established for a large class of
distributions, mainly by exploiting Markov property or/and Gaussian nature
of underlying stochastic processes, see \cite{azencott1}, \cite{bismut}, 
\cite{bolthausen}, \cite{LD techniques by Dembo and Zeitouni}, \cite{Deu-Str}%
, \cite{d-w1}, \cite{donsker-stationary}, \cite{Ellis85}, \cite{Fre-Wen}, 
\cite{Ledoux}, \cite{Mayer-Nua-Pere}, \cite{LDP for FBM} etc. for a small
sample.

Let $\mathbf{W}_{0}^{d}=C_{0}([0,1];R^{d})$ be the Banach space of all
continuous paths in $R^{d}$ started at $0$ with running time $[0,1]$,
equipped with the uniform norm 
\begin{equation*}
||w||=\sup_{t\in \lbrack 0,1]}|w(t)|\text{ \ \ \ \ }\forall w\in \mathbf{W}%
_{0}^{d}\text{.}
\end{equation*}
Let $H=H_{0}^{1}([0,1];R^{d})$ be the subspace of $h\in \mathbf{W}_{0}^{d}$
such that its generalized derivative $\dot{h}\in L^{2}[0,1]$. $H$ is a
Hilbert space under the norm 
\begin{equation*}
||h||_{H^{1}}=\sqrt{\int_{0}^{1}|\dot{h}(t)|^{2}dt}\text{ \ \ \ \ \ }\forall
h\in H\text{.}
\end{equation*}

Let $(w_{t})_{t\geq 0}$ be the coordinate process on $\mathbf{W}_{0}^{d}$ 
\begin{equation*}
w_{t}(x)=x(t)\text{ \ \ \ \ \ \ \ }\forall x\in \mathbf{W}_{0}^{d}\text{ and 
}t\in \lbrack 0,1]
\end{equation*}
and $\mathcal{F}_{t}^{0}=\sigma \{w_{s}:s\leq t\}$ be the filtration
generated by $(w_{t})_{t\geq 0}$. Then $\mathcal{F}_{1}^{0}$ coincides with
the Borel $\sigma $-algebra $\mathcal{B}(\mathbf{W}_{0}^{d})$ on $\mathbf{W}%
_{0}^{d}$. The Wiener measure $P^{w}$ (see for example \cite{ikeda-watanabe1}%
), where the superscript $w$ an attribute to Wiener who first constructed
the law of Brownian motion as a measure on the space of continuous paths, is
the unique probability on $(\mathbf{W}_{0}^{d},\mathcal{B}(\mathbf{W}%
_{0}^{d}))$ such that the coordinate process $(w_{t})$ is a Brownian motion
started from $0$. Another, but equivalent, description of $P^{w}$, is that $%
P^{w}$ is the unique probability on $(\mathbf{W}_{0}^{d},\mathcal{B}(\mathbf{%
W}_{0}^{d}))$ with characteristic function 
\begin{equation*}
\int_{\mathbf{W}_{0}^{d}}e^{\sqrt{-1}l(x)}P^{w}(dx)=e^{-\frac{1}{2}%
||l||_{H^{1}}^{2}}\text{ \ \ \ \ }\forall l\in \left( \mathbf{W}%
_{0}^{d}\right) ^{\ast }
\end{equation*}
where $\left( \mathbf{W}_{0}^{d}\right) ^{\ast }$ is the dual space of $%
\mathbf{W}_{0}^{d}$, and the natural imbedding $\left( \mathbf{W}%
_{0}^{d}\right) ^{\ast }\hookrightarrow H\hookrightarrow \mathbf{W}_{0}^{d}$
has been used.

The Hilbert space $H_{0}^{1}([0,1];R^{d})$ is called the Cameron-Martin
space, and the probability space $(\mathbf{W}_{0}^{d},\mathcal{F}_{1},P^{w})$
is called the Wiener space on $R^{d}$, where $\mathcal{F}_{1}$ is the
completion of $\mathcal{F}_{1}^{0}$ under $P^{w}$, and $(\mathcal{F}%
_{t})_{t\in \lbrack 0,1]}$ is the smallest $\sigma $-algebra containing $%
\mathcal{F}_{t}^{0}$ and the events in $\mathcal{F}_{1}$ with probability
zero. $(\mathcal{F}_{t})_{t\in \lbrack 0,1]}$ is the Brownian filtration.

For each $\varepsilon >0$, $P_{\varepsilon }^{w}$ denotes the distribution
of the scaled Brownian motion $(\sqrt{\varepsilon }w_{t})_{0\leq t\leq 1}$,
that is, $P_{\varepsilon }^{w}$ is the probability measure on $(\mathbf{W}%
_{0}^{d},\mathcal{F}_{1})$ such that 
\begin{eqnarray*}
\int_{\mathbf{W}_{0}^{d}}e^{\sqrt{-1}l(x)}P_{\varepsilon }^{w}(dx) &=&\int_{%
\mathbf{W}_{0}^{d}}e^{\sqrt{-1}l(\sqrt{\varepsilon }x)}P^{w}(dx) \\
&=&e^{-\frac{\varepsilon }{2}||l||_{H^{1}}^{2}}\text{\ \ \ \ }\forall l\in
\left( \mathbf{W}_{0}^{d}\right) ^{\ast }\text{.}
\end{eqnarray*}
It is obvious that, as $\varepsilon \downarrow 0$, $P_{\varepsilon }^{w}$
approaches zero (the probability measure with the support containing only
one path: $x(t)=0$ for all $t$), at an exponential rate. The family of
distributions,\ $\{P_{\varepsilon }^{w}:\varepsilon >0\}$, satisfies the
large deviation principle with rate function 
\begin{equation}
I(h)=\left\{ 
\begin{array}{c}
\frac{1}{2}\int_{0}^{1}|\dot{h}|^{2}(t)dt\text{\ \ \ if }h\in H\text{ ,} \\ 
\infty \text{, \ \ \ \ \ \ \ \ \ \ \ \ \ \ \ \ otherwise.}
\end{array}
\right.  \label{bsdeut2}
\end{equation}
By a large deviation principle with rate function $I$, we mean that 
\begin{equation}
\overline{\lim_{\varepsilon \downarrow 0}}\varepsilon \log P_{\varepsilon
}^{w}(F)\leq -\inf_{w\in F}I(w)  \label{ldf901}
\end{equation}
and 
\begin{equation}
\underline{\lim }_{\varepsilon \downarrow 0}\varepsilon \log P_{\varepsilon
}^{w}(O)\geq -\inf_{w\in O}I(w)  \label{ldf902}
\end{equation}
for any closed subset $F$ and open set $O$. See \cite{LD techniques by Dembo
and Zeitouni}, \cite{Deu-Str} for further information about the general
theory of large deviations.

Brownian motion is a typical example among Markov processes and Gaussian
processes, and is also a good example of continuous martingales. It is thus
natural to seek for large deviation results for laws of properly scaled
martingales. To the best knowledge of the present authors, it remains an
open question whether a large deviation principle holds for martingales, see
however \cite{gra-01}, \cite{g-h01}, \cite{l-d01} and the references therein
for results on exponential tail estimates for martingales in discrete-time.

This article presents a solution of this problem: we are going to establish
a large deviation principle for square-integrable martingales over the
Brownian filtration.

Consider a square-integrable martingale $(Y_{t})_{t\in \lbrack 0,1]}$ (with
initial zero) on $(\mathbf{W}_{0}^{d},\mathcal{F}_{1},\mathcal{F}_{t},P^{w})$%
, then, by the martingale representation theorem (Theorem 3.5, page 201, 
\cite{revuz-yor}), $(Y_{t})_{t\in \lbrack 0,1]}$ is continuous and can be
represented as an It\^{o} integral against Brownian motion, i.e. 
\begin{equation*}
Y_{t}=\int_{0}^{t}f_{s}dw_{s}\text{ \ \ \ \ }
\end{equation*}
where $(f_{t})_{t\geq 0}$ is a predictable process on $(\mathbf{W}_{0}^{d},%
\mathcal{F}_{1},\mathcal{F}_{t},P^{w})$. In particular, $(Y_{t})_{t\geq 0}$
is a measurable function of Brownian motion $(w_{t})_{t\geq 0}$.

It is obvious that the scaling $\sqrt{\varepsilon }$, which is correct for
Brownian motion, does not apply to an arbitrary martingale. Consider the\
Wiener-It\^{o} chaos decomposition (see \cite{ito1}, \cite{wiener1}) of a
square-integrable martingale $Y_{t}=P^{w}(Z_{1}|\mathcal{F}_{t})$ with mean
zero, where $P^{w}(\cdot |\mathcal{F}_{t})$ is the conditional expectation.
Since $Y_{1}\in L^{2}(\mathbf{W}_{0}^{d},\mathcal{F}_{1},P^{w})$, so that
(for simplicity, let us consider the case that $d=1$, but our arguments
equally apply to higher dimensions) 
\begin{equation}
Y_{1}=\sum_{k=1}^{\infty }\int_{0<t_{1}<\cdots <t_{k}<1}f_{k}(t_{1},\cdots
,t_{k})dw_{t_{1}}\cdots dw_{t_{k}}  \label{kl01}
\end{equation}
where $k$-th term, a multiple Wiener-It\^{o} integral, belongs to $k$-th
Wiener chaos, and the integrands $f_{k}$ are symmetric functions in $%
L^{2}[0,1]^{k}$. Clearly 
\begin{equation*}
Y_{t}=\sum_{k=1}^{\infty }\int_{0<t_{1}<\cdots <t_{k}<t}f_{k}(t_{1},\cdots
,t_{k})dw_{t_{1}}\cdots dw_{t_{k}}\text{.}
\end{equation*}
According to Schilder's theorem, one simple while reasonable re-scaling for
such martingale is multiplying $k$-th term in the decomposition by $%
\varepsilon ^{k/2}$. Therefore, one possible scaling for martingales should
be $P_{-\log \sqrt{\varepsilon }}Y_{t}$, where $(P_{t})_{t\geq 0}$ is the
Ornstein-Uhlenbeck semigroup on $(\mathbf{W}_{0}^{d},\mathcal{F}_{1},P^{w})$%
. This is the first place the Ornstein-Uhlenbeck semigroup comes into our
study.

Let us define the following mapping $F:H_{0}^{1}([0,1];R^{1})\rightarrow 
\mathbf{W}_{0}^{1}$ by 
\begin{equation*}
F(h)_{t}=\sum_{k=1}^{\infty }\int_{0<t_{1}<\cdots
<t_{k}<t}f_{k}(t_{1},\cdots ,t_{k})\dot{h}(t_{1})\cdots \dot{h}%
(t_{k})dt_{1}\cdots dt_{k}\text{.}
\end{equation*}
The main result of the paper is the following large deviation principle.

\begin{theorem}
\label{LD for Martingale}Let $\xi \in L^{2}(\mathbf{W}_{0}^{1},\mathcal{F}%
_{1},P^{w})$ which has the Wiener-It\^{o}'s decomposition 
\begin{equation*}
\xi =\sum_{k=1}^{\infty }\int_{0<t_{1}<\cdots <t_{k}<t}f_{k}(t_{1},\cdots
,t_{k})dw_{t_{1}}\cdots dw_{t_{k}}\text{ .}
\end{equation*}
Let $Y_{t}=P^{w}\left( \xi |\mathcal{F}_{t}\right) $ for $t\in \lbrack 0,1]$%
, and $\nu _{\varepsilon }$ be the distribution of $(P_{-\log \sqrt{%
\varepsilon }}Y_{t})_{t\leq 1}$. Then $\{\nu _{\varepsilon }:\varepsilon \in
(0,1)\}$\ satisfies the large deviation principle with rate function 
\begin{equation*}
I^{\prime }(w)=\inf \left\{ I(h)\mid h\in H\text{ such that }F(h)=w\right\} 
\text{,}
\end{equation*}
where $I(h)=\frac{1}{2}||h||_{H^{1}}^{2}$ for $h\in H$.\ \ 
\end{theorem}

A special case of the above theorem, namely, for distributions of multiple
Wiener-It\^{o}'s integrals of the following form 
\begin{equation*}
Y_{t}=\int_{0<t_{1}<\cdots <t_{k}<t}f(t_{1},\cdots ,t_{k})dw_{t_{1}}\cdots
dw_{t_{k}}
\end{equation*}
has been established in M. Ledoux \cite{Ledoux}. \ Nualart and etl. \cite
{Mayer-Nua-Pere} extended to an even larger class of multiple
Wiener-It\^{o}'s integrals on an abstract Wiener space. These authors used
the same scaling $\varepsilon ^{\frac{k}{2}}$ but only for single iterated
integrals, we however believe that their methods may be modified to develop
a large deviation principle for finite sums of multiple Wiener-It\^{o}
integrals, or more precisely for the laws of martingales of the following
form 
\begin{equation*}
Y_{t}^{\varepsilon }=\sum_{k=1}^{N}\varepsilon ^{\frac{k}{2}%
}\int_{0<t_{1}<\cdots <t_{k}<t}f_{k}(t_{1},\cdots ,t_{k})dw_{t_{1}}\cdots
dw_{t_{k}}\text{.}
\end{equation*}

Our study is based on the following simple observation: the scaling for $k$%
-th term $\varepsilon ^{\frac{k}{2}}$ is itself sub-exponential for large $k$%
, which ensures the laws of martingales 
\begin{equation*}
Y_{t}^{\varepsilon }=\sum_{k=1}^{\infty }\varepsilon ^{\frac{k}{2}%
}\int_{0<t_{1}<\cdots <t_{k}<t}f_{k}(t_{1},\cdots ,t_{k})dw_{t_{1}}\cdots
dw_{t_{k}}\text{,}
\end{equation*}
though the sum is infinite, remain to satisfy the large deviation principle.

In this stage we would like to describe the main steps of our proof of
Theorem \ref{LD for Martingale}, which are necessary long and involve many
technical issues. The first step is, of course, to approximate $%
Y_{t}^{\varepsilon }$ by a good family of martingales. More precisely, let $%
\xi \in L^{2}(\mathbf{W}_{0}^{d},\mathcal{F}_{1},P^{w})$ have the
decomposition 
\begin{equation*}
\xi =\sum_{k=1}^{\infty }\int_{0<t_{1}<\cdots <t_{k}<1}f_{k}(t_{1},\cdots
,t_{k})dw_{t_{1}}\cdots dw_{t_{k}}
\end{equation*}
so that 
\begin{equation*}
||\xi ||_{L^{2}}^{2}=\sum_{k=1}^{\infty }\frac{1}{k!}%
||f_{k}||_{L^{2}[0,1]^{k}}^{2}<\infty \text{.}
\end{equation*}
For each $n$ we may choose an $N_{n}$ and 
\begin{equation*}
\xi _{n}=\sum_{k=1}^{N_{n}}\int_{0<t_{1}<\cdots
<t_{k}<1}f_{k}^{n}(t_{1},\cdots ,t_{k})dw_{t_{1}}\cdots dw_{t_{k}}\text{.}
\end{equation*}
such that $\xi _{n}\rightarrow \xi $ in $L^{2}(\mathbf{W}_{0}^{d},\mathcal{F}%
_{1},P^{w})$. The symmetric functions $f_{k}^{n}$ may be chosen such that $%
f_{k}^{n}\rightarrow f_{k}$ in $L^{2}$ for each $k$ as $n\rightarrow \infty $%
. We can require that all $f_{k}^{n}$ are smooth enough with bounded
derivatives (up to order $4$ is enough) on $[0,1]$, and moreover, we can
assume that $f_{k}^{n}$ have a product form 
\begin{equation*}
f_{k}^{n}(t_{1},\cdots ,t_{n})=\sum_{j_{1},\cdots
,j_{k}=1}^{N_{n}}C_{j_{1},\cdots ,j_{k}}^{n,k}f_{j_{1}}^{k}(t_{1})\cdots
f_{j_{k}}^{k}(t_{k})
\end{equation*}
where $C_{j_{1},\cdots ,j_{k}}^{n,k}$ are constants and $N_{n}$ are natural
numbers. Thanks to the hypercontractivity of the Ornstein--Uhlenbeck
semigroup (see L. Gross \cite{gross1} for more details), the corresponding
martingales $Y(n)_{t}^{\varepsilon }=P^{w}(\xi _{n}|\mathcal{F}_{t})$
converges to $Y^{\varepsilon }$ exponentially.

The next step is to show that $Y(n)^{\varepsilon }$ for each $n$ satisfies
the large deviation principle, and to identify its rate function explicitly,
which will be achieved by using Lyons' continuity theorem (\cite{lyons1},
see also \cite{lq1}, and excellent recent books \cite{Friz-Victor}, \cite
{lyons-etc} etc.), Schilder's large deviation principle in $p$-variation
distance (see \cite{lqz}) together with a simple application of Varadhan's
contraction principle.\ More precisely we demonstrate that, for each $n$, $%
Y(n)^{\varepsilon }$ may be realized (or more precisely lifted) as a
continuous function on the space of geometric rough paths, with respect to a
variation distance. This is the precise version of what belonging to the
folklore that Startonovich's integrals are continuous functions of Brownian
motion paths. However we should emphasize that the continuity here must be
understood in terms of Lyons' $p$-variation distance, rather than the
uniform norm, see Proposition \ref{prop-m} below for a more precise
statement.

Nevertheless, it turns out that the rate function governing the large
deviations of $\{Y(n)^{\varepsilon }:\varepsilon \in (0,1)\}$ is given by 
\begin{equation*}
I_{n}(w)=\inf \{I(h):F_{n}(h)=w\}
\end{equation*}
where 
\begin{equation*}
F_{n}(h)_{t}=\sum_{k=1}^{N_{n}}\int_{0<t_{1}<\cdots
<t_{k}<t}f_{k}^{n}(t_{1},\cdots ,t_{k})\dot{h}(t_{1})\cdots \dot{h}%
(t_{k})dw_{t_{1}}\cdots dw_{t_{k}}\text{.}
\end{equation*}

It is easy to see that $F_{n}$ converges uniformly on any level set of $I$
with respect to the uniform norm, but in the uniform norm $F_{n}$ is not
continuous from $H\subset \mathbf{W}_{0}^{d}$ to $\mathbf{W}_{0}^{d}$.
Indeed there is no continuous extension of $F_{n}$ to the whole space $%
\mathbf{W}_{0}^{d}$ in general. On the other hand, we may lift the mappings $%
F_{n}$ to the space of rough paths, that is, $F_{n}$ is continuous in the $p$%
-variation distance, but then in general $F_{n}$ does not converge with
respect to the $p$-variation metric as we do not have control over the
derivatives of the integrands $f_{k}$ ($k=1,2,\cdots $). Therefore, the
existed (extended or generalized) contraction principles, which require that 
$F_{n}$ are continuous and $F_{n}$ converges uniformly on level sets of $I$,
do not apply to the present case to deduce a large deviation principle for $%
\{Y^{\varepsilon }:\varepsilon \in (0,1)\}$.

The main technical tool established in Section 2, a continuity theorem for
large deviations which we believe has independent interest by its own,
however, allows us to prove the large deviation principle for $%
\{Y^{\varepsilon }:\varepsilon \in (0,1)\}$. In Section 3, we show that the
hypercontractivity of the Ornstein-Uhlenbeck operator allows us to establish
the exponential tightness of the family of scaled martingales, which is one
of the main ingredients in our proof of the main result. In Section 4, we
construct the It\^{o}-Lyons mappings associated with multiple Wiener-It\^{o}
integrals, which thus makes another key step towards the proof of Theorem 
\ref{LD for Martingale}. Finally in Section 5, we collect all technical
estimates together to establish a large deviation principle for
square-integrable martingales.

\section{A continuity theorem for large deviations}

An important method in the theory of large deviations is the contraction
principle, formulated by S. R. S. Varadhan \cite{Var84}. Suppose $%
\{Z^{\varepsilon }:\varepsilon \in (0,1)\}$ is a family of random variables
in a Polish space $E$ which satisfies the large deviation principle with a
good rate function $I$, and suppose $F:E\rightarrow E^{\prime }$ is a
continuous mapping, where $E^{\prime }$ is another Polish space, then $%
\{X^{\varepsilon }:\varepsilon \in (0,1)\}$, where $X^{\varepsilon
}=F(Z^{\varepsilon })$, also satisfies the large deviation principle with
rate function 
\begin{equation*}
I^{\prime }(s^{\prime })=\inf \left\{ I(s):s\in E\text{ \ such that }%
F(s)=s^{\prime }\text{ }\right\}
\end{equation*}
for any $s^{\prime }\in E^{\prime }$.

However, in stochastic analysis, we often deal with Wiener functionals, for
example, strong solutions to stochastic differential equations, which are
only measurable rather than continuous, the above contraction principle is
not sufficient in applications. Different generalizations of the contraction
principle, proposed by various authors over the past years, have been
successfully applied to distributions of many interesting Wiener
functionals. Among these generalizations, a typical one may be formulated as
the following (see Theorem 4.2.23 in \cite{LD techniques by Dembo and
Zeitouni}).\ Suppose $F_{n}:E\rightarrow E^{\prime }$ is a family of \emph{%
continuous} mappings, and $\{X^{\varepsilon }:\varepsilon \in (0,1)\}$ is a
family of random variables in $E^{\prime }$ on $(\Omega ,\mathcal{F},P)$,
such that the continuous images $F_{n}(Z^{\varepsilon })$ approaches $%
X^{\varepsilon }$ in probability at an exponential rate 
\begin{equation}
\overline{\lim_{n\rightarrow \infty }}\varepsilon \log P\left\{ \rho
^{\prime }(F_{n}(Z^{\varepsilon }),X^{\varepsilon })\right\} =-\infty
\label{exp01}
\end{equation}
where $\rho $ and $\rho ^{\prime }$ are the distance functions on $E$ and $%
E^{\prime }$ respectively. In addition, if $F_{n}$ \emph{converges uniformly}
on any level set $K_{L}\equiv \{s:I(s)\leq L\}$, its limit is denoted by $F$
(note that $F$ is only well defined on the effective set $H\equiv
\{s:I(s)<\infty \}$, but $F$ is continuous on $H\subset E$), then the
distributions of $\{X^{\varepsilon }:\varepsilon \in (0,1)\}$ obeys the
large deviation principle with rate function 
\begin{equation}
I^{\prime }(s^{\prime })=\inf \left\{ I(s):s\in H\text{ \ such that }%
F(s)=s^{\prime }\text{ }\right\}  \label{exp02}
\end{equation}
for any $s^{\prime }\in E^{\prime }$.

In general we are interested in the following question. Suppose $%
\{X_{n}^{\varepsilon }:\varepsilon \in (0,1)\}$ is a sequence of families of
random variables in $E$ on a probability space $(\Omega ,\mathcal{F},P)$
which converges to $\{X^{\varepsilon }:\varepsilon \in (0,1)\}$
exponentially 
\begin{equation}
\overline{\lim_{n\rightarrow \infty }}\varepsilon \log P\left\{ \rho
(X_{n}^{\varepsilon },X^{\varepsilon })\right\} =-\infty \text{.}
\label{exp04}
\end{equation}
Suppose for each $n$, $\{X_{n}^{\varepsilon }:\varepsilon \in (0,1)\}$
satisfies a large deviation principle with rate function $I_{n}$. Then,
according to Theorem 4.2.16, page 131, \cite{LD techniques by Dembo and
Zeitouni}, the limiting distribution $\mu _{\varepsilon }$ of $%
X^{\varepsilon }$ satisfies a weak large deviation principle with rate
function 
\begin{equation}
I_{\infty }(s^{\prime })\equiv \sup_{\delta >0}\lim \inf_{n\rightarrow
\infty }\inf_{s\in B(s_{0}^{\prime },\delta )}I_{n}(s)  \label{ex-07}
\end{equation}
where $B(s_{0}^{\prime },\delta )$ is open ball centered at $s_{0}^{\prime }$
with radius $\delta $. Furthermore, if in addition $I_{\infty }$ is a good
rate function and for any closed set $S$ one has 
\begin{equation}
\inf_{s^{\prime }\in S}I_{\infty }(s^{\prime })\leq \lim \sup_{n\rightarrow
\infty }\inf_{s\in S}I_{n}(s)  \label{ex-08}
\end{equation}
then $\{\mu _{\varepsilon }:\varepsilon \in (0,1)\}$ satisfies the large
deviation principle with rate function $I_{\infty }$.

However, in general, one can not deduce that a large deviation principle for
limit processes $\{X^{\varepsilon }:\varepsilon \in (0,1)\}$ under (\ref
{exp04}) alone.

In many applications, the rate functions $I_{n}$ are often given as the
images of a common good rate function $I$ under some mappings $%
F_{n}:H=\{I<\infty \}\rightarrow E$. If $F_{n}$ are not continuous with
respect to the distance on $E$, then the right-hand side of (\ref{ex-07}) is
difficult to compute, and it is hard then to verify the condition (\ref
{ex-08}). The main goal of this section is to provide useful sufficient
conditions in this situation, such that the limiting processes $%
\{X^{\varepsilon }:\varepsilon \in (0,1)\}$ satisfies a large deviation
principle. To this end, we introduce a concept of rate-function mappings,
see Definition \ref{r-f mapping} below, which are mappings sending a good
rate function to another one.

More precisely, we handle the following situation. Suppose $%
X_{n}^{\varepsilon }$, $X^{\varepsilon }$ are random variables in $E$, $%
X_{n}^{\varepsilon }$ converges to $X^{\varepsilon }$ exponentially, i.e. (%
\ref{exp04}) is satisfied, and for each $n$, $X_{n}^{\varepsilon }$
satisfies the large deviation principle with rate function given by 
\begin{equation*}
I_{n}(s^{\prime })=\inf \{I(s):F_{n}(s)=s^{\prime }\}
\end{equation*}
where $I$ is a good rate function, and $F_{n}$ are rate-function mappings,
so that each $I_{n}$ is a good rate function. In many interesting cases, $%
F_{n}$ are not continuous in the topology on $E$. To ensure a large
deviation principle of limiting processes $X^{\varepsilon }$ to hold, the
main conditions we impose on the family of rate-function mappings $\{F_{n}\}$
are the followings. $F_{n}$ converges uniformly on any level set $%
\{s:I(s)\leq L\}$, and all $F_{n}$ are weakly continuous in a proper
topology on the effective set $H\equiv \{s:I(s)<\infty \}$. For more
details, see Theorem \ref{limit-LDP} below.

\subsection{Rate-function mappings}

Let $E$ be a separable Banach space with its norm denoted by $||\cdot ||$.
The induced distance function is denoted by $\rho $, that is, $\rho
(s,s^{\prime })=$ $||s-s^{\prime }||$. Let $I$ be a good rate function on $E$%
, that is, $I:E\rightarrow \lbrack 0,\infty ]$ such that for each real
number $L\geq 0$, its level set $K_{L}=\{s:I(s)\leq L\}$ is compact in $E$.

Let $H=\{s:I(s)<\infty \}$ be the effective set of the definition of $I$. We
assume that

\begin{enumerate}
\item  $H$ is a dense vector subspace of $E$, and there is a Hilbert norm $%
||\cdot ||_{H}$ on $H$, such that $(H,||\cdot ||_{H})$ is a Hilbert space.

\item  For each real $L\geq 0$, $K_{L}$ is weakly compact, bounded and
closed in $(H,||\cdot ||_{H})$.
\end{enumerate}

It is necessary that for all $s\in H$, $||s||\leq C||s||_{H}$ for some
constant $C$.

The aim of this part is to study a class of mappings $F:H\rightarrow E$ so
that 
\begin{equation}
I_{F}(s^{\prime })=\inf \{I(s):s\in H\text{ such that }F(s)=s^{\prime }\}
\label{rt-1}
\end{equation}
is again a good rate function on $E$.

\begin{definition}
\label{r-f mapping}A mapping $F:H\rightarrow E$ is a called a rate-function
mapping, if the following conditions are satisfied.

1) $F$ is continuous with respect to the corresponding norms, i.e. 
\begin{equation*}
||F(s)-F(s^{\prime })||\rightarrow 0\text{ \ \ as }||s-s^{\prime
}||_{H}\rightarrow 0\text{.}
\end{equation*}
Note that $F$ may be not continuous as a mapping $H\subset E$ to $E$ with
respect to the norm $||\cdot ||$.

2) $F:H\rightarrow E$ is weakly continuous on any $K_{L}$ in the following
sense: if $s_{k}\rightarrow s$ weakly in $H$, where $s_{k}\in K_{L}$ (so
that $s\in K_{L}$ as well), then $F(s_{k})\rightarrow F(s)$ weakly in $E$.

3) For any $L\geq 0$, the range $F(K_{L})=\{s^{\prime }:s\in K_{L}$ such
that $F(s)=s^{\prime }\}$ is compact in $(E,||\cdot ||)$.
\end{definition}

The following is the main use of the concept of rate-function mappings.

\begin{proposition}
\label{prop-r1}If $F:H\rightarrow E$ is a rate-function mapping, then $I_{F}$
defined by (\ref{rt-1}) is a good rate function on $E$.
\end{proposition}

\begin{proof}
Let $K_{L}^{\prime }=\{s^{\prime }\in E:I_{F}(s^{\prime })\leq L\}$. We have
to show that $K_{L}^{\prime }$ is a compact subset of $E$. To this end,
choose any sequence $\{s_{n}^{\prime }\}$ in $K_{L}^{\prime }$, such that $%
I_{F}(s_{n}^{\prime })\leq L$. Then there are $s_{n}\in H$ such that $%
F(s_{n})=s_{n}^{\prime }$ and $I(s_{n})\leq L+\frac{1}{n}$. In particular $%
s_{n}\in K_{L+1}$, and $\{s_{n}^{\prime }\}\subset $ $F(K_{L+1})$ which is
compact in $E$. Therefore we may assume that $s_{n}^{\prime }\rightarrow
s_{0}^{\prime }$ in $E$, otherwise consider a convergent subsequence
instead. Since $\{s_{n}\}\subset K_{L+1}$ so that we can extract a
subsequence $s_{n_{k}}\rightarrow s_{0}$ weakly in $H$ as well as $%
s_{n_{k}}\rightarrow s_{0}$ $\ $in $E$. Since $F$ is weakly continuous on $%
K_{L+1}$, so that $F(s_{n_{k}})\rightarrow F(s_{0})$ weakly in $E$.
Therefore we must have $F(s_{0})=s_{0}^{\prime }$ and, since $I$ is a good
rate function, so that 
\begin{equation*}
I_{F}(s_{0}^{\prime })\leq I(s_{0})\leq \overline{\lim }_{n\rightarrow
\infty }I(s_{n})=L
\end{equation*}
which implies that $s_{0}^{\prime }\in K_{L}^{\prime }$. Therefore $%
K_{L}^{\prime }$ is compact.
\end{proof}

The following proposition shows that the set of rate-function mappings is
closed under the uniform convergence on level sets of $I$.

\begin{proposition}
\label{lem-i1}Let $F_{n}:H\rightarrow E$ be a sequence of rate-function
mappings.\ Suppose that $F_{n}$ converges uniformly on $K_{L}$ for any $%
L\geq 0$, and let $F$ denote the limiting function. Then $F$ is also a
rate-function mapping.
\end{proposition}

\begin{proof}
As the uniform limit, $F$ is clearly continuous from $(H,||\cdot ||_{H})$ to 
$(E,||\cdot ||)$. To show that $F$ is weakly continuous on $K_{L}$, consider
any $s_{k},s\in K_{L}$, $s_{k}\rightarrow s$ weakly in $H$. Since $%
F_{n}\rightarrow F$ uniformly on $K_{L}$, for every $\epsilon >0$, there is
an $N_{1}$ such that 
\begin{equation}
||F_{n}(s)-F(s)||<\frac{\varepsilon }{3}\text{ \ \ \ \ \ \ }\forall n\geq
N_{1}\text{ \ }\forall s\in K_{L}\text{.}  \label{eea1}
\end{equation}
Let $\xi \in E^{\ast }$. Then 
\begin{eqnarray*}
\left| \langle \xi ,F(s_{k})-F(s)\rangle \right| &\leq &\left| \langle \xi
,F(s_{k})-F_{n}(s_{k})\rangle \right| +\left| \langle \xi
,F_{n}(s)-F(s)\rangle \right| \\
&&+\left| \langle \xi ,F_{n}(s_{k})-F_{n}(s)\rangle \right| \\
&\leq &\frac{2\varepsilon }{3}||\xi ||_{E^{\ast }}+\left| \langle \xi
,F_{n}(s_{k})-F_{n}(s)\rangle \right| \\
&\rightarrow &\frac{2\varepsilon }{3}||\xi ||_{E^{\ast }}\text{ \ \ \ \ \ \
\ \ \ }\forall n\geq N_{1}
\end{eqnarray*}
as $k\rightarrow \infty $, so that $F$ is weakly continuous on $K_{L}$.

Next we prove that $F(K_{L})$ is compact in $E$. Consider any sequence $%
\{s_{k}^{\prime }\}\subset F(K_{L})$, so that $F(s_{k})=s_{k}^{\prime }$ for
some $s_{k}\in K_{L}$. For any $\epsilon >0$, there is an $N_{1}$ such that (%
\ref{eea1}) holds. Hence 
\begin{equation*}
||F(s_{k})-F(s_{l})||\leq \frac{2\varepsilon }{3}%
+||F_{N_{1}}(s_{k})-F_{N_{1}}(s_{l})||\text{ \ \ \ }\forall k,l\text{ .}
\end{equation*}
Since $F_{N_{1}}(K_{L})$ is compact, we may assume that $\{F_{N_{1}}(s_{k})%
\} $ is convergent, so that there is an $N_{2}$ such that 
\begin{equation*}
||F_{N_{1}}(s_{k})-F_{N_{1}}(s_{l})||\leq \frac{\varepsilon }{3}\text{ \ \ }%
\forall k,l\geq N_{2}
\end{equation*}
and therefore 
\begin{equation*}
||F(s_{k})-F(s_{l})||\leq \varepsilon \text{ \ \ \ \ \ \ }\forall k,l\geq
N_{1}\vee N_{2}\text{.}
\end{equation*}
Hence $F(s_{k})\rightarrow s^{\prime }$ in $E$ for some $s^{\prime }\in E$.
We need to show that $s^{\prime }\in F(K_{L})$.

Since $K_{L}$ is a compact subset of $E$ and is weakly compact in $H$, we
may assume that $\rho (s_{k},s_{0})\rightarrow 0$ for some $s_{0}\in K_{L}$,
and $s_{k}\rightarrow s_{0}$ weakly in $H$ as well, otherwise considering a
convergent subsequence instead. By (\ref{eea1}) we have 
\begin{equation*}
||F_{n}(s_{k})-F(s_{k})||\leq \frac{\varepsilon }{3}\text{ \ \ \ \ \ }%
\forall n\geq N_{1}
\end{equation*}
so that, for any $\xi \in E^{\ast }$ 
\begin{equation*}
|\langle \xi ,F_{n}(s_{k})-F(s_{k})\rangle |\leq \frac{\varepsilon }{3}||\xi
||_{E^{\ast }}\text{\ \ \ \ }\forall n\geq N_{1}\text{.}
\end{equation*}
Letting $k\rightarrow \infty $, then $\langle \xi ,F_{n}(s_{k})\rangle
\rightarrow \langle \xi ,F_{n}(s_{0})\rangle $ and $f(s_{k})\rightarrow
s^{\prime }$ so that 
\begin{equation*}
|\langle \xi ,F_{n}(s_{0})-s^{\prime }\rangle |\leq \frac{\varepsilon }{3}%
||\xi ||_{E^{\ast }}\text{\ \ \ \ }\forall n\geq N_{1}\text{.}
\end{equation*}
Letting $n\rightarrow \infty $ in the above inequality, to obtain 
\begin{equation*}
|\langle \xi ,F(s_{0})-s^{\prime }\rangle |\leq \frac{\varepsilon }{3}||\xi
||_{E^{\ast }}
\end{equation*}
for any $\epsilon >0$ and $\xi \in E^{\ast }$. Therefore we must have $%
F(s_{0})=s^{\prime }$, so that $s^{\prime }\in F(K_{L})$.
\end{proof}

We end this sub-section by showing some examples of rate-function mappings.

\begin{proposition}
\label{prop-r3}Let $E=C_{0}([0,1];R^{1})$ be the Banach space of all
continuous paths in $R^{1}$ starting from zero, endowed with the uniform
norm $||s||=\sup_{t\in \lbrack 0,1]}|s(t)|$, and $H$ be the subspace of all
paths $s\in E$ which has a generalized derivative $\dot{s}\in L^{2}[0,1]$,
together with the Hilbert norm $||s||_{H}=\sqrt{\int_{0}^{1}|\dot{s}%
(t)|^{2}dt}$.\ Then $H$ is a Hilbert space which is dense in $E$. Let $I(s)=%
\frac{1}{2}||s||_{H}^{2}$ if $s\in H$, otherwise $I(s)=\infty $. Then $I$ is
a good rate function on $E$ with effective set $H$.

Let $f_{n}\in L^{2}[0,1]^{n}$ be symmetric functions ($n=1,2,\cdots $) such
that 
\begin{equation}
\sum_{n=1}^{\infty }\frac{1}{n!}||f_{n}||_{L^{2}[0,1]^{n}}^{2}<\infty \text{.%
}  \label{nprm-1}
\end{equation}
Define $F_{N}$ and $F:H\rightarrow E$ by 
\begin{equation}
F_{N}(h)_{t}=\sum_{n=1}^{N}\int_{0<t_{1}<\cdots <t_{n}<t}f_{n}(t_{1},\cdots
,t_{n})\dot{h}(t_{1})\cdots \dot{h}(t_{n})dt_{1}\cdots dt_{n}  \label{f-r4}
\end{equation}
and 
\begin{equation}
F(h)_{t}=\sum_{n=1}^{\infty }\int_{0<t_{1}<\cdots
<t_{n}<t}f_{n}(t_{1},\cdots ,t_{n})\dot{h}(t_{1})\cdots \dot{h}%
(t_{n})dt_{1}\cdots dt_{n}\text{ }  \label{f-r3}
\end{equation}
\ \ for $t\in \lbrack 0,1]$, respectively. Then

1) For any $L\geq 0$, $F_{N}$ converges to $F$ uniformly on $K_{L}$ in $%
(E,||\cdot ||)$. That is 
\begin{equation}
\sup_{h\in K_{L}}||F_{N}(h)-F(h)||\rightarrow 0\text{ \ \ as \ \ }%
N\rightarrow \infty \text{.}  \label{f-r5}
\end{equation}

2) All $F_{N}$, $F$ are rate-function mappings.
\end{proposition}

\begin{proof}
Let $K_{L}=\{h\in E:I(h)\leq L\}$ be the level set of the rate function $I$.
First of all, we note that for each $L\geq 0$, $K_{L}$ is a closed ball in $%
H $, and therefore $K_{L}$ is not only compact in $E$ (by the Sobolev
imbedding), $K_{L}$ is also convex and bounded in $H$, so that $K_{L}$ is
weakly compact in $H$, according to Milman's theorem and Theorem 1, page
126, \cite{yosida}.

It is easy to see that each $F_{N}$ is continuous from $(H,||\cdot ||_{H})$
to $(E,||\cdot ||)$. Let us prove that $F_{N}\rightarrow F$ uniformly on any 
$K_{L}$. Let $h\in K_{L}$. Then, by the Cauchy-Schwartz inequality 
\begin{eqnarray*}
&&\left| \sum_{n=N+1}^{\infty }\int_{0<t_{1}<\cdots
<t_{n}<t}f_{n}(t_{1},\cdots ,t_{n})\dot{h}(t_{1})\cdots \dot{h}%
(t_{n})dt_{1}\cdots dt_{n}\right| \\
&\leq &\sqrt{\sum_{n=N+1}^{\infty }\int_{0<t_{1}<\cdots
<t_{n}<t}f_{n}(t_{1},\cdots ,t_{n})^{2}dt_{1}\cdots dt_{n}} \\
&&\times \sqrt{\sum_{n=N+1}^{\infty }\int_{0<t_{1}<\cdots <t_{n}<t}|\dot{h}%
(t_{1})\cdots \dot{h}(t_{n})|^{2}dt_{1}\cdots dt_{n}} \\
&=&\sqrt{\sum_{n=N+1}^{\infty }\frac{1}{n!}||f_{n}||_{L^{2}[0,1]^{n}}^{2}}%
\sqrt{\sum_{n=N+1}^{\infty }\frac{1}{n!}||\dot{h}||_{H^{1}}^{2n}} \\
&\leq &\sqrt{\sum_{n=N+1}^{\infty }\frac{(2L)^{n}}{n!}}\sqrt{%
\sum_{n=N+1}^{\infty }\frac{1}{n!}||f_{n}||_{L^{2}[0,1]^{n}}^{2}}
\end{eqnarray*}
where we have used the fact that, if $g$ is a symmetric function on $%
[0,t]^{k}$, then 
\begin{equation*}
\int_{0<t_{1}<\cdots <t_{k}<t}g(t_{1},\cdots ,t_{k})dt_{1}\cdots dt_{k}=%
\frac{1}{k!}\int_{[0,t]^{k}}g(t_{1},\cdots ,t_{k})dt_{1}\cdots dt_{k}
\end{equation*}
as long as $g$ is integrable.

Therefore 
\begin{equation*}
\sup_{h\in K_{L}}\sup_{t\leq 1}\left| F(h)_{t}-F_{N}(h)_{t}\right|
\rightarrow 0\text{ \ as }N\rightarrow \infty
\end{equation*}
which proves our claim.

Let us prove that $F_{N}$ is weakly continuous on $K_{L}$ as stated in the
lemma. To show the weak continuity of $F_{N}$, we only need to show the weak
continuity of $F_{N}\equiv \Psi $ which has a simple form, namely 
\begin{equation*}
\Psi (h)_{t}=\int_{0<t_{1}<\cdots <t_{n}<t}f(t_{1},\cdots ,t_{n})\dot{h}%
(t_{1})\cdots \dot{h}(t_{n})dt_{1}\cdots dt_{n}
\end{equation*}
where $f\in L^{2}[0,1]^{n}$ which is symmetric, and has the following form 
\begin{equation*}
f(t_{1},\cdots ,t_{n})=\sum_{j_{1},\cdots ,j_{n}}^{m}C^{j_{1},\cdots
,j_{n}}f_{j_{1}}(t_{1})\cdots f_{j_{n}}(t_{n})
\end{equation*}
where $f_{j_{k}}\in L^{2}[0,1]$. Let $h_{k}\rightarrow h$ weakly in $H$,
where $h_{k}\in K_{L}$ (so that $h\in K_{L}$). Since $f$ is symmetric, we
thus have 
\begin{eqnarray*}
&&\Psi (h_{k})_{t}-\Psi (h)_{t} \\
&=&\frac{1}{n!}\int_{[0,t]^{n}}f(t_{1},\cdots ,t_{n})\left( \dot{h}%
_{k}(t_{1})\cdots \dot{h}_{k}(t_{n})-\dot{h}(t_{1})\cdots \dot{h}%
(t_{n})\right) dt_{1}\cdots dt_{n} \\
&=&\frac{1}{n!}\sum_{j_{1},\cdots ,j_{n}}^{m}C^{j_{1},\cdots ,j_{n}} \\
&&\times \int_{\lbrack 0,t]^{n}}f_{j_{1}}(t_{1})\cdots
f_{j_{n}}(t_{n})\left( \dot{h}_{k}(t_{1})\cdots \dot{h}_{k}(t_{n})-\dot{h}%
(t_{1})\cdots \dot{h}(t_{n})\right) dt_{1}\cdots dt_{n} \\
&=&\frac{1}{n!}\sum_{j_{1},\cdots ,j_{n}}^{m}C^{j_{1},\cdots
,j_{n}}\sum_{l=1}^{n-1}\langle 1_{[0,t]}f_{j_{1}},h_{k}\rangle \cdots
\langle 1_{[0,t]}f_{j_{n-j+1}},h_{k}\rangle \langle
1_{[0,t]}f_{j_{n-j}},h_{k}-h\rangle \\
&&\times \langle 1_{[0,t]}f_{j_{n-j-1}},h\rangle \cdots \langle
1_{[0,t]}f_{j_{n}},h\rangle
\end{eqnarray*}
where $\langle f,h\rangle =\int_{0}^{1}f(t)\dot{h}(t)dt$, which yields that 
\begin{eqnarray*}
&&\left| \Psi (h_{k})_{t}-\Psi (h)_{t}\right| \\
&\leq &\frac{1}{n!}\sum_{j_{1},\cdots ,j_{n}}^{m}|C^{j_{1},\cdots
,j_{n}}|\left( \sqrt{2L}\right) ^{n-1}||f_{j_{1}}||_{L^{2}[0,1]}\cdots
||f_{j_{n}}||_{L^{2}[0,1]} \\
&&\times \left| \langle 1_{[0,t]}f_{j_{n-j}},h_{k}-h\rangle \right| \text{.}
\end{eqnarray*}
Therefore $\left| \Psi (h_{k})_{t}-\Psi (h)_{t}\right| \rightarrow 0$ for
any $t\in \lbrack 0,1]$, as $h_{k}\rightarrow h$ weakly in $H$, and $\{\Psi
(h_{k})\}$ is bounded uniformly as $\{h_{k}\}\subset K_{L}$, so that $\Psi
(h_{k})\rightarrow \Psi (h)$ weakly in $E$.

Therefore all $F_{N}$ $F$ are weakly continuous.

For any $h\in K_{L}$ and $[s,t]\subset \lbrack 0,1]$, we set 
\begin{equation*}
\Delta _{\lbrack s,t]}^{n}=[0,t]^{n}\setminus \lbrack 0,s]^{n}\text{.}
\end{equation*}
Then 
\begin{equation*}
F(h)_{t}-F(h)_{s}=\sum_{n=1}^{\infty }\frac{1}{n!}\int_{\Delta _{\lbrack
s,t]}^{n}}f_{n}(t_{1},\cdots ,t_{n})\dot{h}(t_{1})\cdots \dot{h}%
(t_{n})dt_{1}\cdots dt_{n}
\end{equation*}
so that, by utilizing the Cauchy-Schwarz inequality 
\begin{eqnarray*}
|F(h)_{t}-F(h)_{s}| &\leq &\sqrt{\sum_{n=1}^{\infty }\frac{1}{n!}%
\int_{\Delta _{\lbrack s,t]}^{n}}|f_{n}(t_{1},\cdots
,t_{n})|^{2}dt_{1}\cdots dt_{n}}\sqrt{\sum_{n=1}^{\infty }\frac{1}{n!}%
||h||_{H1}^{2n}} \\
&\leq &e^{L}\sqrt{\sum_{n=1}^{\infty }\frac{1}{n!}\int_{\Delta _{\lbrack
s,t]}^{n}}|f_{n}(t_{1},\cdots ,t_{n})|^{2}dt_{1}\cdots dt_{n}}\text{.}
\end{eqnarray*}
Since 
\begin{eqnarray*}
&&\sum_{n=1}^{\infty }\frac{1}{n!}\int_{\Delta _{\lbrack
s,t]}^{n}}|f_{n}(t_{1},\cdots ,t_{n})|^{2}dt_{1}\cdots dt_{n} \\
&\leq &\sum_{n=1}^{\infty }\frac{1}{n!}\int_{[0,1]^{n}}|f_{n}(t_{1},\cdots
,t_{n})|^{2}dt_{1}\cdots dt_{n} \\
&<&\infty
\end{eqnarray*}
and for each $n$, according to the Lebesgue theorem 
\begin{equation*}
\int_{\Delta _{\lbrack s,t]}^{n}}|f_{n}(t_{1},\cdots
,t_{n})|^{2}dt_{1}\cdots dt_{n}\rightarrow 0\text{ as }s\uparrow t\text{,}
\end{equation*}
therefore we can conclude that the functions in $F(K_{L})$ are
equi-continuous on $[0,1]$, and are bounded in $E$: 
\begin{equation*}
|F(h)_{t}|\leq e^{L}\sqrt{\sum_{n=1}^{\infty }\frac{1}{n!}%
||f||_{L^{2}[0,1]^{n}}^{2}}\text{ \ \ }\forall h\in K_{L}\text{.}
\end{equation*}
Therefore, according to Ascoli-Arzel\`{a}'s theorem (page 85, Section III-3, 
\cite{yosida}). $F(K_{L})$ is pre-compact.\ We now need to show that $%
F(K_{L})$ is closed in $E$. Let $\{w_{n}\}$ be any sequence in $F(K_{L})$
which converges to $w$ in $E$. Let $h_{n}\in K_{L}$ such that $%
F(h_{n})=w_{n} $. $K_{L}$ is weakly compact in $H$, so let us assume that $%
h_{n}\rightarrow h$ weakly in $H$, otherwise consider a weakly convergent
subsequence instead. Since $K_{L}$ is a closed and convex subset of $H$, so
that $h\in K_{L}$. Therefore $F(h_{n})=w_{n}$ weakly converges to $F(h)$ in $%
E$. We thus must have $F(h)=w$, so that $w\in F(K_{L})$, and $F(K_{L})$ is
compact.
\end{proof}

Of course, similar results hold in higher dimensions, where $%
E=C_{0}([0,1];R^{d})$, $H=H_{0}^{1}([0,1];R^{d})$ and the rate function $%
I(h)=\frac{1}{2}||h||_{H^{1}}^{2}$.

\subsection{Continuity of large deviations}

We have thus developed necessary tools to formulate a continuity theorem for
large deviation principles.

\begin{theorem}
\label{limit-LDP}Suppose $H\subset E$ and $I$ is a good rate function
satisfying the two conditions listed at the beginning of the last subsection
2.1. Let $F_{n}:H\rightarrow E$ be a sequence of rate-function mappings, and
suppose that $F_{n}$ converges to $F$ uniformly on any level set $%
K_{L}=\{s:I(s)\leq L\}$. For each $n$, let $\{X_{n}^{\varepsilon
}:\varepsilon \in (0,1)\}$ (as well as $\{X^{\varepsilon }:\varepsilon \in
(0,1)\}$) be a family of random variables valued in $E$ on a complete
probability space $(\Omega ,\mathcal{F},P)$. Suppose the following
conditions are satisfied.

1) $\{X_{n}^{\varepsilon }:\varepsilon \in (0,1)\}$ converges to $%
\{X^{\varepsilon }:\varepsilon \in (0,1)\}$ exponentially: for any $\delta
>0 $%
\begin{equation}
\overline{\lim_{n\rightarrow \infty }}\varepsilon \log P\left\{ \rho
(X_{n}^{\varepsilon },X^{\varepsilon })>\delta \right\} =-\infty \text{.}
\label{exa1}
\end{equation}

2) For each $n$, $\{X_{n}^{\varepsilon }:\varepsilon \in (0,1)\}$ satisfies
the large deviation principle with rate function $I_{F_{n}}$.

Then, the distribution family $\{\mu _{\varepsilon }:\varepsilon \in (0,1)\} 
$ of the limiting process $\{X^{\varepsilon }:\varepsilon \in (0,1)\}$
satisfies the large deviation principle with rate function $I_{F}$.
\end{theorem}

\begin{proof}
For simplicity, we use $I_{n}^{\prime }$ to denote $I_{F_{n}}$ and $%
I^{\prime }$ for $I_{F}$. Let $\rho $ be the distance function on $E$, i.e. 
\begin{equation*}
\rho (s,s^{\prime })=||s-s^{\prime }||\text{ \ \ \ \ }\forall s,s^{\prime
}\in E\text{,}
\end{equation*}
and for $s_{0}^{\prime }\in E$ and $\delta >0$, $B(s_{0}^{\prime },\delta )$
denote the open ball in $E$ centered at $s_{0}^{\prime }$ with radius $%
\delta $.

Firstly we show the lower bound. Let $O$ be an open subset of $E$, we need
to prove that 
\begin{equation}
\underline{\lim }_{\varepsilon \downarrow 0}\varepsilon \log \mu
_{\varepsilon }(O)\geq -\inf_{s\in O}I^{\prime }(s)\text{ .}
\label{CPLDP open part}
\end{equation}
It is easy to see that we only need to show 
\begin{equation}
\underline{\lim }_{\varepsilon \downarrow 0}\varepsilon \log \mu _{\epsilon
}(B(s_{0}^{\prime },\delta ))\geq -\inf_{s^{\prime }\in B(s_{0}^{\prime
},\delta /2)}I^{\prime }(s^{\prime })\text{.}  \label{CPLDP open enough}
\end{equation}
for any $s_{0}^{\prime }\in O$ and $\delta >0$ such that $B(s_{0}^{\prime
},\delta )\subset O$. We may assume that $\inf_{s^{\prime }\in
B(s_{0}^{\prime },\delta /2)}I^{\prime }(s^{\prime })<\infty $, otherwise
there is nothing to prove. By the triangle inequality one has for any $%
\lambda >0$%
\begin{eqnarray*}
P\left\{ \rho (X_{n}^{\varepsilon },s_{0}^{\prime })<\frac{\delta }{3}%
\right\} &\leq &P\left\{ \rho (X_{n}^{\varepsilon },X^{\varepsilon
})>\lambda \right\} \\
&&+P\left\{ \rho (X^{\varepsilon },s_{0}^{\prime })<\lambda +\frac{\delta }{3%
}\right\}
\end{eqnarray*}
it follows that 
\begin{eqnarray*}
&&\log P\left\{ \rho (X_{n}^{\varepsilon },s_{0}^{\prime })<\frac{2\delta }{3%
}\right\} \\
&\leq &\log 2+\log \left\{ P\left\{ \rho (X_{n}^{\varepsilon
},X^{\varepsilon })>\lambda \right\} \vee P\left\{ \rho (X^{\varepsilon
},s_{0}^{\prime })<\lambda +\frac{2\delta }{3}\right\} \right\} \\
&\leq &\log 2+\log P\left\{ \rho (X_{n}^{\varepsilon },X^{\varepsilon
})>\lambda \right\} \vee \log P\left\{ \rho (X^{\varepsilon },s_{0}^{\prime
})<\lambda +\frac{2\delta }{3}\right\} \text{.}
\end{eqnarray*}
Therefore 
\begin{eqnarray*}
&&\varepsilon \log P\left\{ \rho (X_{n}^{\varepsilon },s_{0}^{\prime })<%
\frac{2\delta }{3}\right\} \\
&\leq &\varepsilon \log 2 \\
&&+\max \left\{ \varepsilon \log P\left\{ \rho (X_{n}^{\varepsilon
},X^{\varepsilon })>\lambda \right\} ;\varepsilon \log P\left\{ \rho
(X^{\varepsilon },s_{0}^{\prime })<\lambda +\frac{2\delta }{3}\right\}
\right\}
\end{eqnarray*}
hence 
\begin{eqnarray*}
&&\lim_{n\rightarrow \infty }\underline{\lim }_{\varepsilon \downarrow
0}\varepsilon \log P\left\{ \rho (X_{n}^{\varepsilon },s_{0}^{\prime })<%
\frac{2\delta }{3}\right\} \\
&\leq &\max \left\{ \lim_{n\rightarrow \infty }\underline{\lim }%
_{\varepsilon \downarrow 0}\varepsilon \log P\left\{ \rho
(X_{n}^{\varepsilon },X^{\varepsilon })>\lambda \right\} ;\underline{\lim }%
_{\varepsilon \downarrow 0}\varepsilon \log P\left\{ \rho (X^{\varepsilon
},s_{0}^{\prime })<\lambda +\frac{2\delta }{3}\right\} \right\} \\
&=&\underline{\lim }_{\varepsilon \downarrow 0}\varepsilon \log P\left\{
\rho (X^{\varepsilon },s_{0}^{\prime })<\lambda +\frac{2\delta }{3}\right\}
\end{eqnarray*}
for any $\lambda >0$, we have used the assumption that 
\begin{equation*}
\lim_{n\rightarrow \infty }\underline{\lim }_{\varepsilon \downarrow
0}\varepsilon \log P\left\{ \rho (X_{n}^{\varepsilon },X^{\varepsilon
})>\lambda \right\} =0\text{.}
\end{equation*}
On the other hand, as $\{X_{n}^{\varepsilon }:\varepsilon \in (0,1)\}$
satisfies the large deviation principle with rate function $I_{N}^{\prime }$%
, so that 
\begin{equation*}
\underline{\lim }_{\varepsilon \downarrow 0}\varepsilon \log P\left\{ \rho
(X_{n}^{\varepsilon },s_{0}^{\prime })<\frac{2\delta }{3}\right\} \geq
-\inf_{s^{\prime }\in B(s_{0}^{\prime },\frac{2\delta }{3})}I_{n}^{\prime
}(s^{\prime })
\end{equation*}
and therefore 
\begin{equation}
\underline{\lim }_{\varepsilon \downarrow 0}\varepsilon \log P\left\{ \rho
(X^{\varepsilon },s_{0}^{\prime })<\lambda +\frac{2\delta }{3}\right\} \geq
-\lim_{n\rightarrow \infty }\inf_{s^{\prime }\in B(s_{0}^{\prime },\frac{%
2\delta }{3})}I_{n}^{\prime }(s^{\prime })  \label{ft-01}
\end{equation}
for any $\lambda >0$.

According to the assumption that $\inf_{s^{\prime }\in B(s_{0}^{\prime },%
\frac{\delta }{2})}I^{\prime }(s^{\prime })=M<\infty $. Since $I^{\prime }$
is a good rate function, there is an $s_{1}^{\prime }\in B(s_{0}^{\prime },%
\frac{7}{12}\delta )$, such that $I^{\prime }(s_{1}^{\prime })=M$. Since 
\begin{equation*}
I^{\prime }(s^{\prime })=\inf \{I(s)\mid s\in H\text{ such that }%
F(s)=s^{\prime }\}
\end{equation*}
and since $I$ is a good rate function on $E$, there is an $s_{1}\in H$ such
that $I(s_{1})=M$ and $F(s_{1})=s_{1}^{\prime }$. Let $t_{n}^{\prime
}=F_{n}(s_{1})\in E$. Then $\lim_{n\rightarrow \infty }t_{n}^{\prime
}=F(s_{1})=s_{1}^{\prime }$, so that for every $\alpha >0$ there exists an $%
N_{0}$ such that 
\begin{equation*}
t_{n}^{\prime }\in B(s_{1}^{\prime },\alpha )\text{ \ \ \ }\forall n>N_{0}%
\text{.}
\end{equation*}
Then for $n>N_{0}$ we have 
\begin{eqnarray*}
\inf_{s^{\prime }\in B(s_{1}^{\prime },\alpha )}I_{n}^{\prime }(s^{\prime })
&\leq &I_{n}^{\prime }(t_{n}^{\prime }) \\
&=&\inf \{I(s):s\in H\text{ and }F_{n}(s)=t_{n}^{\prime }\} \\
&\leq &I(s_{1})\text{.}
\end{eqnarray*}
Choose $\alpha =\frac{1}{24}\delta $. Then 
\begin{eqnarray*}
\inf_{s^{\prime }\in B(s_{0}^{\prime },\frac{2\delta }{3})}I_{n}^{\prime
}(s^{\prime }) &\leq &I(s_{1})=M \\
&=&\inf_{s^{\prime }\in B(s_{0}^{\prime },\frac{\delta }{2})}I^{\prime
}(s^{\prime })\text{ \ \ \ \ \ }
\end{eqnarray*}
so that 
\begin{equation*}
\lim_{n\rightarrow \infty }\inf_{s^{\prime }\in B(s_{0}^{\prime },\frac{%
2\delta }{3})}I_{n}^{\prime }(s^{\prime })\leq \inf_{s^{\prime }\in
B(s_{0}^{\prime },\frac{\delta }{2})}I^{\prime }(s^{\prime })\text{ \ \ \ \
\ }\forall n>N_{0}\text{.}
\end{equation*}
Hence 
\begin{equation}
\underline{\lim }_{\varepsilon \downarrow 0}\varepsilon \log P\left\{ \rho
(X^{\varepsilon },s_{0}^{\prime })<\lambda +\frac{2\delta }{3}\right\} \geq
-\inf_{s^{\prime }\in B(s_{0}^{\prime },\frac{\delta }{2})}I^{\prime
}(s^{\prime })  \label{rd1}
\end{equation}
for any $\lambda >0$, which implies (\ref{CPLDP open enough}).

Now prove the upper bound: for any closed set $S$ in $E$, 
\begin{equation}
\overline{\lim }_{\varepsilon \downarrow 0}\varepsilon \log \mu
_{\varepsilon }(S)\leq -\inf_{s^{\prime }\in F}I^{\prime }(s^{\prime })\text{%
.}  \label{CPLDP closed part}
\end{equation}
For any $\delta >0$, set 
\begin{equation*}
F^{\delta }=\{s^{\prime }\in E\mid \text{ }s^{\prime \prime }\in S\text{
s.t. }\rho (s^{\prime },s^{\prime \prime })<\delta \}\text{.}
\end{equation*}
Then 
\begin{equation*}
P\left\{ X^{\varepsilon }\in S\right\} \leq P\left\{ X_{n}^{\varepsilon }\in
S^{\delta }\right\} +P\left\{ \rho (X^{\varepsilon },X_{n}^{\varepsilon
})>\delta \right\}
\end{equation*}
and therefore 
\begin{equation*}
\log P\left\{ X^{\varepsilon }\in S\right\} \leq \log 2+\log \left[ P\left\{
X_{n}^{\varepsilon }\in S^{\delta }\right\} \vee P\left\{ \rho
(X^{\varepsilon },X_{n}^{\varepsilon })>\delta \right\} \right]
\end{equation*}
so that, for any $\delta >0$%
\begin{eqnarray*}
&&\overline{\lim }_{\varepsilon \downarrow 0}\varepsilon \log P\left\{
X^{\varepsilon }\in S\right\} \\
&\leq &\left[ \overline{\lim }_{\varepsilon \downarrow 0}\varepsilon \log
P\left\{ X_{n}^{\varepsilon }\in S^{\delta }\right\} \vee \overline{\lim }%
_{\varepsilon \downarrow 0}\varepsilon \log P\left\{ \rho (X^{\varepsilon
},X_{n}^{\varepsilon })>\delta \right\} \right] \text{.}
\end{eqnarray*}

For any $K>0$ there is an $N_{2}$ depending only on $\delta $ and $K$, such
that 
\begin{equation*}
\overline{\lim }_{\varepsilon \downarrow 0}\varepsilon \log P\left\{ \rho
(X^{\varepsilon },X_{n}^{\varepsilon })>\delta \right\} \leq -K\text{ \ \ }%
\forall n>N_{2}\text{ .}
\end{equation*}
On the other hand, $\{X_{n}^{\varepsilon }:\varepsilon \in (0,1)\}$
satisfies the large deviation principle with rate function $I_{n}^{\prime }$%
, so that 
\begin{eqnarray*}
&&\overline{\lim }_{\varepsilon \downarrow 0}\varepsilon \log P\left\{
X^{\varepsilon }\in S\right\} \\
&\leq &(-K)\vee \overline{\lim }_{\varepsilon \downarrow 0}\varepsilon \log
P\left\{ X_{n}^{\varepsilon }\in S^{\delta }\right\} \\
&\leq &\max \left\{ -\inf_{\overline{S^{\delta }}}I_{n}^{\prime },-K\right\} 
\text{ \ \ \ \ \ \ \ \ }\forall n>N_{2}\text{. }
\end{eqnarray*}
It follows that 
\begin{equation}
\overline{\lim }_{\varepsilon \downarrow 0}\varepsilon \log P\left\{
X^{\varepsilon }\in S\right\} \leq -\lim_{\delta \downarrow 0}\underline{%
\lim }_{n\rightarrow \infty }\inf_{\overline{S^{\delta }}}I_{n}^{\prime }%
\text{ .}  \label{ht-01}
\end{equation}
Let us consider 
\begin{equation*}
l=\lim_{\delta \downarrow 0}\underline{\lim }_{n\rightarrow \infty }\inf_{%
\overline{S^{\delta }}}I_{n}^{\prime }\text{ .}
\end{equation*}
If $l=\infty $, then 
\begin{eqnarray*}
\overline{\lim }_{\varepsilon \downarrow 0}\varepsilon \log P\left\{
X^{\varepsilon }\in S\right\} &=&-\infty \\
&\leq &-\inf_{s^{\prime }\in S}I^{\prime }(s^{\prime })
\end{eqnarray*}
so let us assume that $l<\infty $. In this case we show that 
\begin{equation}
\inf_{F}I^{\prime }\leq l=\lim_{\delta \downarrow 0}\underline{\lim }%
_{n\rightarrow \infty }\inf_{\overline{S^{\delta }}}I_{n}^{\prime }\text{ .}
\label{tr-01}
\end{equation}
In this case, by definition of the multiple limits of the right-hand side of
(\ref{tr-01}) we may choose a sequence $(s_{m})\subset K_{l+1}$, and a
subsequence $n_{m}\rightarrow \infty $ such that 
\begin{equation*}
f_{n_{m}}(s_{m})=s_{m}^{\prime }\text{, \ \ }\rho (s_{m}^{\prime },S)\leq 
\frac{1}{m}
\end{equation*}
and 
\begin{equation*}
I(s_{m})\leq l+\frac{1}{m}\text{ .}
\end{equation*}
Then $\{s_{m}\}\subset K_{l+1}$. Since $K_{l+1}$ is compact in $E$, and
weakly compact in $H$, we can further assume that $s_{m}\rightarrow s$ in $%
K_{l+1}$ (in the distance $\rho $), and $s_{m}\rightarrow s$ weakly in $H$.
Since $I$ is lower semi-continuous, $I(s)\leq l$.

For any $\alpha >0$, there is a number $N_{3}$ such that 
\begin{equation}
\rho (F_{n_{m}}(s),F(s))<\frac{\alpha }{2}\text{ \ \ \ \ \ }\forall s\in
K_{l+1}  \label{tr-03}
\end{equation}
for any $m\geq N_{3}$. In particular 
\begin{equation*}
\rho (s_{m}^{\prime },F(s_{m}))<\frac{\alpha }{2}\text{\ \ \ \ }\forall
m\geq N_{3}\text{.}
\end{equation*}
However $\{F(s_{m}):m=1,2,\cdots \}\subset F(K_{l+1})$ which is compact in $%
E $. Therefore, if necessary by extracting a subsequence, we may assume $%
\{F(s_{m})\}$ converges in $E$ to $s^{\prime }$.\ Hence, there is an $N_{4}$
such that 
\begin{equation*}
\rho (F(s_{m}),s^{\prime })<\frac{\alpha }{2}\text{ \ \ \ \ \ }\forall m\geq
N_{4}\text{ .}
\end{equation*}
Therefore 
\begin{equation*}
\rho (s_{m}^{\prime },s^{\prime })<\alpha \text{ \ \ \ \ \ }\forall m\geq
N_{3}\vee N_{4}\text{ .}
\end{equation*}
That is, $s_{m}^{\prime }\rightarrow s^{\prime }$ in $E$, so that $s^{\prime
}\in F$. \ On the other hand, $s_{m}\rightarrow s$ weakly in $H$, so that,
as $F$ is weakly continuous on $K_{l+1}$, $F(s_{m})\rightarrow F(s)$ weakly
in $E$. We thus must have $F(s)=s^{\prime }\in F$ and $I(s)\leq l$.
Therefore $\inf_{F}I^{\prime }\leq l$ which completes the proof of (\ref
{tr-01}).
\end{proof}

\section{Hypercontractivity and martingales}

Let us retain the notations we have established in Introduction. In
particular, $(\mathbf{W}_{0}^{d},\mathcal{F}_{1},P^{w})$ is the Wiener space
on $R^{d}$. However, for simplicity, we may assume that $d=1$ without loss
of generality.

If $f\in L^{2}[0,1]^{n}$ we use 
\begin{equation*}
J_{n}(f)_{t}=\int_{0<t_{1}<\cdots <t_{n}<t}f(t_{1},\cdots
,t_{n})dw_{t_{1}}\cdots dw_{t_{n}}
\end{equation*}
to denote the multiple Wiener-It\^{o} integral on $[0,t]$, $t\in \lbrack
0,1] $. $\{J_{n}(f)_{t}\}$ is a square-integrable martingale up to time $1$.
According to Wiener-It\^{o}'s chaos decomposition (\cite{ito1}, \cite
{wiener1}), if $\xi \in L^{2}(\mathbf{W}_{0}^{d},\mathcal{F}_{1},P^{w})$,
then 
\begin{equation*}
\xi =E\xi +\sum_{n=1}^{\infty }J_{n}(f_{n})_{1}
\end{equation*}
for a sequence of symmetric functions $f_{n}\in L^{2}[0,1]^{n}$ and 
\begin{equation*}
||\xi -E\xi ||_{2}^{2}=\sum_{n=1}^{\infty }\frac{1}{n!}%
||f_{n}||_{L^{2}[0,1]^{n}}^{2}
\end{equation*}
where $||\xi ||_{p}$ denotes the $L^{p}$-norm of $\xi $. The
Ornstein-Uhlenbeck semigroup $(P_{t})_{t\geq 0}$ is defined by 
\begin{equation*}
P_{t}\xi =E(\xi )+\sum_{n=1}^{\infty }e^{-nt}J_{n}(f_{n})_{1}\text{.}
\end{equation*}
$(P_{t})_{t\geq 0}$ is a symmetric diffusion semigroup on $L^{2}(\mathbf{W}%
_{0}^{d},\mathcal{F}_{1},P^{w})$, which may be extended uniquely to a
strongly continuous semigroup on $L^{p}(\mathbf{W}_{0}^{d},\mathcal{F}%
_{1},P^{w})$ for every $p\geq 1$.

The following hypercontractivity of the Ornstein-Uhlenbeck semigroup plays a
major rule in this paper.

\begin{theorem}
\label{gross01}(L. Gross, Nelson \cite{gross1}) The Ornstein-Uhlenbeck
semigroup $(P_{t})_{t\geq 0}$ possesses the hypercontractivity 
\begin{equation*}
||P_{t}\xi ||_{p(t)}\leq ||\xi ||_{p}
\end{equation*}
for all $\xi \in L^{2}(\mathbf{W}_{0}^{d},\mathcal{F}_{1},P^{w})$, $p>1$ and 
$t>0$, where $p(t)=1+(p-1)e^{2t}$.
\end{theorem}

As an application of the hypercontractivity, we present a proof of the
following estimate, a well-known result in Gaussian analysis, which shows
that tail behaviors of multiple Wiener-It\^{o} integrals.

\begin{proposition}
\label{lemmal1}Let $\xi =I_{n}(f)_{1}$where $f\in L^{2}([0,1]^{n})$ for some 
$n$. Let $Y_{t}=P^{w}\left( \xi |\mathcal{F}_{t}\right) $ for $t\in \lbrack
0,1]$. Then for any $\alpha <n/(2e)$ 
\begin{equation}
E\exp \left( \alpha \left| \frac{1}{||\xi ||_{2}}\sup_{t\leq 1}Y_{t}\right|
^{2/n}\right) \leq C_{\alpha ,n}  \label{es-k1}
\end{equation}
where 
\begin{equation*}
C_{\alpha ,n}=1+4e^{\alpha }+\frac{2e}{\sqrt{2\pi }}\sum_{k\geq n}\frac{1}{%
\sqrt{k}}\left( \frac{2\alpha e}{n}\right) ^{k}\text{.}
\end{equation*}
Therefore, for any $\delta >0$ 
\begin{equation}
P^{w}\left\{ \left| \sup_{t\leq 1}Y_{t}\right| \geq \delta \right\} \leq
C_{\alpha ,n}\exp \left\{ -\alpha \frac{\delta ^{2/n}}{||\xi ||_{2}^{2/n}}%
\right\} \text{.}  \label{expl01}
\end{equation}
\end{proposition}

\begin{proof}
Without losing generality, we may assume that $||\xi ||_{2}=1$. According to
the hypercontractivity of $(P_{t})$, $P_{t}\xi =e^{-nt}\xi \in L^{p(t)}$
where $p(t)=1+e^{2t}$, and 
\begin{equation*}
\left( E\left| \xi \right| ^{1+e^{2t}}\right) ^{1/(1+e^{2t})}\leq e^{nt}%
\text{ \ \ \ \ \ \ \ }\forall t>0
\end{equation*}
that is for any $p>1$, $\xi \in L^{p}(\mathbf{W}_{0}^{d},\mathcal{F}%
_{1},P^{w})$, and 
\begin{equation*}
E\left| \xi \right| ^{p}\leq (p-1)^{np/2}\text{ \ \ \ }\forall p>1\text{.}
\end{equation*}
By Doob's inequality, $\sup_{t\leq 1}Y_{t}\in L^{p}(\mathbf{W}_{0}^{d},%
\mathcal{F}_{1},P^{w})$ and 
\begin{eqnarray*}
E\left| \sup_{t\leq 1}Y_{t}\right| ^{p} &\leq &\left( \frac{p}{p-1}\right)
^{p}(p-1)^{np/2} \\
&<&ep(p-1)^{\frac{np}{2}-1}
\end{eqnarray*}
for any $p>1$. Since 
\begin{eqnarray*}
E\exp \left( \alpha \left| \sup_{t\leq 1}Y_{t}\right| ^{\theta }\right)
&=&\sum_{k=0}\frac{\alpha ^{k}}{k!}E\left| \sup_{t\leq 1}Y_{t}\right|
^{\theta k} \\
&=&1+\sum_{k\theta \leq 2}\frac{\alpha ^{k}}{k!}E\left| \sup_{t\leq
1}Y_{t}\right| ^{k\theta }+\sum_{k\theta \geq 2}\frac{\alpha ^{k}}{k!}%
E\left| \sup_{t\leq 1}Y_{t}\right| ^{k\theta } \\
&\leq &1+\sum_{k\theta \leq 2}\frac{\alpha ^{k}}{k!}\left( E\left|
\sup_{t\leq 1}Y_{t}\right| ^{2}\right) ^{\theta k/2} \\
&&+e\sum_{k\theta \geq 2}\frac{\alpha ^{k}}{k!}\frac{k\theta }{k\theta -1}%
(k\theta -1)^{nk\theta /2} \\
&\leq &1+\sum_{k\theta \leq 2}\frac{\alpha ^{k}}{k!}4^{\theta
k/2}+e\sum_{k\theta \geq 2}\frac{\alpha ^{k}}{k!}\frac{k\theta }{k\theta -1}%
(k\theta -1)^{nk\theta /2}\text{,}
\end{eqnarray*}
choosing $\theta =2/n$, we thus have 
\begin{equation*}
E\exp \left( \alpha \left| \sup_{t\leq T}Y_{t}\right| ^{2/n}\right) \leq
1+\sum_{k\leq n}\frac{\alpha ^{k}}{k!}4^{k/n}+2e\sum_{k\theta \geq 2}\frac{%
k^{k}}{k!}\left( \frac{2\alpha }{n}\right) ^{k}\text{.}
\end{equation*}
According to Stirling's formula 
\begin{eqnarray*}
\frac{k^{k}}{k!} &\leq &\frac{1}{\sqrt{2\pi }}\frac{e^{k}}{\sqrt{k}}\frac{1}{%
e^{\frac{1}{12k+1}}} \\
&\leq &\frac{1}{\sqrt{2\pi }}\frac{e^{k}}{\sqrt{k}}
\end{eqnarray*}
(see page 52, W. Feller \cite{feller1}) which follows that 
\begin{equation*}
E\exp \left( \alpha \left| \sup_{t\leq T}Y_{t}\right| ^{2/n}\right) \leq
1+4e^{\alpha }+\frac{2e}{\sqrt{2\pi }}\sum_{k\theta \geq 2}\frac{1}{\sqrt{k}}%
\left( \frac{2\alpha e}{n}\right) ^{k}
\end{equation*}
the right-hand is finite if $\alpha <n/(2e)$.
\end{proof}

\begin{proposition}
\label{lemn01}If $\xi \in L^{2}\left( \mathbf{W}_{0}^{d},\mathcal{F}%
_{1},P^{w}\right) $ and $Y_{t}=P^{w}\left( \xi |\mathcal{F}_{t}\right) $,
then for every $\varepsilon \in (0,1)$ and $\delta >0$ 
\begin{equation}
P^{w}\left\{ \left| \sup_{t\leq 1}\left( P_{-\log \sqrt{\varepsilon }%
}Y_{t}\right) \right| \geq \delta \right\} \leq (1+\varepsilon )^{1+\frac{1}{%
\varepsilon }}\frac{||\xi ||_{2}^{1+\frac{1}{\varepsilon }}}{\delta ^{1+%
\frac{1}{\varepsilon }}}\text{.}  \label{dfg04r}
\end{equation}
\end{proposition}

\begin{proof}
By the previous lemma, $P_{-\log \sqrt{\varepsilon }}\xi \in L^{1+\frac{1}{%
\varepsilon }}\left( \mathbf{W}_{0}^{d},\mathcal{F}_{1},P^{w}\right) $ for
any $\varepsilon \in (0,1)$, thus, by Doob's $L^{p}$-inequality 
\begin{eqnarray*}
E\left| \sup_{t\leq T}\left( P_{-\log \sqrt{\varepsilon }}Y_{t}\right)
\right| ^{1+\frac{1}{\varepsilon }} &\leq &(1+\varepsilon )^{1+\frac{1}{%
\varepsilon }}E\left| P_{-\log \sqrt{\varepsilon }}\xi \right| ^{1+\frac{1}{%
\varepsilon }} \\
&\leq &(1+\varepsilon )^{1+\frac{1}{\varepsilon }}||\xi ||_{2}^{1+\frac{1}{%
\varepsilon }}
\end{eqnarray*}
the second inequality follows from the hypercontractivity of the
Ornstein-Uhlenbeck semigroup $(P_{t})_{t\geq 0}$. Therefore 
\begin{eqnarray*}
P^{w}\left\{ \left| \sup_{t\leq T}\left( P_{-\log \sqrt{\varepsilon }%
}Y_{t}\right) \right| \geq \delta \right\} &\leq &\frac{1}{\delta ^{1+\frac{1%
}{\varepsilon }}}E\left| \sup_{t\leq T}\left( P_{-\log \sqrt{\varepsilon }%
}Y_{t}\right) \right| ^{1+\frac{1}{\varepsilon }} \\
&\leq &(1+\varepsilon )^{1+\frac{1}{\varepsilon }}\frac{||\xi ||_{2}^{1+%
\frac{1}{\varepsilon }}}{\delta ^{1+\frac{1}{\varepsilon }}}\text{.}
\end{eqnarray*}
\end{proof}

\section{It\^{o}'s mappings defined by It\^{o}'s multiple integrals}

The large deviation principle for multiple Wiener-It\^{o} integrals has been
established in M. Ledoux \cite{Ledoux}, also in \cite{Mayer-Nua-Pere}. We
believe their arguments, with a little bit of extra work, can equally apply
to the case of finite sum of multiple It\^{o}'s integrals. For completeness
we however include a different proof, which we believe has independent
interest by its own.

Our approach is to apply the contraction principle to the It\^{o}-Lyons
mappings on spaces of geometric rough paths. Not like the original It\^{o}'s
mappings defined by solving stochastic differential equations via It\^{o}'s
calculus, It\^{o}-Lyons mappings will serve the same aim as that of It\^{o}
mappings, but in addition they are continuous with respect to variation
distances. The main concept and the continuity result were established in an
important work by T. Lyons \cite{lyons1} (see also \cite{lq1}, the excellent
recent books \cite{lyons-etc}, \cite{Friz-Victor} etc), which says solutions
to Stratonovich type stochastic differential equations are continuous
functions of Brownian motion paths together with its L\'{e}vy area. A more
precise statement, see items 1 and 2 in Theorem \ref{hgj01} below.

Lyons' continuity theorem, or called the universal limit theorem as
suggested by Malliavin, has been finding many applications in analyzing
Wiener functionals, for example, see the recent articles by Hambly and Lyons 
\cite{hambly-lyons}, Cass and Friz \cite{cass-friz} and etc. The important
fact here is that, the rough path analysis, as developed in \cite{lq1},
allows us more effectively to apply classical functional analytic techniques
to stochastic analysis. The result in this section is another example of the
power of this new analysis.

\subsection{Schilder's theorem in the $p$-variation topology}

In M. Ledoux, Z. Qian and T. Zhang \cite{lqz}, a version of the large
deviation principle of Schilder's in the $p$-variation topology has been
established, with which we will prove the large deviation principle for
martingales.

Let $p\in (2,3)$ be a fixed constant. Let $\mathbb{W}$ be the space of all
continuous path $w\in \mathbf{W}_{0}^{d}$ which has finite total variations
over $[0,1]$: 
\begin{equation*}
\sup_{D}\sum_{l}|w_{t_{l}}-w_{t_{l-1}}|<+\infty
\end{equation*}
where $D$ runs over all finite partitions $\{0=t_{0}<t_{1}<\cdots <t_{n}=1\}$
of the interval $[0,1]$. For a path $w\in \mathbb{W}$ we may consider its
increment $w_{s,t}^{1}=w_{t}-w_{s}$ and its L\'{e}vy area 
\begin{equation*}
w_{s,t}^{2}=\int_{s<t_{1}<t_{2}<t}dw_{t_{1}}\otimes dw_{t_{2}}
\end{equation*}
defined via Riemann sum limits. $w^{2}$ can be considered as a $d\times d$
matrix-valued function on $\Delta \equiv \{(s,t):0\leq s\leq t\leq 1\}$.
Then define 
\begin{equation*}
\mathbf{w}_{s,t}=(1,w_{s,t}^{1},w_{s,t}^{2})\text{ \ \ \ \ \ \ if \ }%
(s,t)\in \Delta
\end{equation*}
and $\mathbf{w}:(s,t)\in \Delta \rightarrow \mathbf{w}_{s,t}$ which is
called the rough path associated to $w\in \mathbb{W}$, a path of finite
variations. The space of all such rough paths is denoted by $\mathbb{W}%
^{\infty }$ (and we may thus identify $\mathbb{W}$ with its ``lift'' $%
\mathbb{W}^{\infty }$), equipped with a natural metric $d_{p}$ (called the $%
p $-variation metric where $p\in (2,3)$) 
\begin{equation}
d_{p}(\mathbf{w},\mathbf{y})=\sup_{D}\left(
\sum_{l}|w_{t_{l-1},t_{l}}^{1}-y_{t_{l-1},t_{l}}^{1}|^{p}\right) ^{\frac{1}{p%
}}+\sup_{D}\left(
\sum_{l}|w_{t_{l-1},t_{l}}^{2}-y_{t_{l-1},t_{l}}^{2}|^{p/2}\right) ^{\frac{2%
}{p}}\text{.}  \label{pvad01}
\end{equation}
Since any $w\in H_{0}^{1}([0,1];R^{d})$ has a finite variation on $[0,1]$,
therefore the Cameron-Martin space $H_{0}^{1}([0,1];R^{d})$ may be
considered as a subspace of $\mathbb{W}^{\infty }$, hence of $\mathbb{W}^{p}$
to be introduced later on.

The completion of $\mathbb{W}^{\infty }$ under the $p$-variation metric $%
d_{p}$ is denoted by $\mathbb{W}^{p}$. T. Lyons \cite{lyons1} has
established the following result. Consider the ordinary differential
equation 
\begin{equation}
dy_{t}^{i}=f_{0}^{i}(t,y_{t})dt+\sum_{j=1}^{d}f_{j}^{i}(t,y_{t})\circ
dw_{t}^{j}\text{, \ \ \ }y_{0}=0  \label{odef01}
\end{equation}
$i=1,\cdots ,m$, where we have used $\circ dw_{t}^{j}$ to denote the usual
differential if $w$ is differentiable, to indicate (\ref{odef01}) should be
understood as Stratonovich stochastic differential equations for Brownian
motion $w$. Bot interpretation of (\ref{odef01}) within the setting of rough
path analysis.

If $f_{j}^{i}$, $f_{0}^{i}$ are $C_{b}^{3}$ functions, then $\mathbf{w}%
\rightarrow \mathbf{y}$ is continuous map from $\mathbb{W}^{\infty }$ into $%
\mathbb{W}^{\infty }$ under $p$-variation metric $d_{p}$ and therefore
extended continuously to be a map from $\mathbb{W}^{p}$ into $\mathbb{W}^{p}$%
, called the It\^{o}-Lyons mapping determined by (\ref{odef01}). This
result, together with the following theorem proved in \cite{lq1} and Ledoux,
Qian and Zhang \cite{lqz}, can be used to establish large deviation
principles for a large class of It\^{o}'s functionals.

\begin{theorem}
\label{hgj01}Let $p\in (2,3)$ be a fixed a constant. Let $(\mathbf{W}%
_{0}^{d},\mathcal{F}_{1},P^{w})$ be the $d$-dimensional Wiener space, so
that its coordinate process $(w_{t})_{t\in \lbrack 0,1]}$ is an $R^{d}$%
-valued Brownian motion. Let $2<p<3$ be a fixed constant. Set 
\begin{equation*}
w_{s,t}^{1}=w_{t}-w_{s}
\end{equation*}
and 
\begin{equation*}
w_{s,t}^{2}=\int_{s<t_{1}<t_{2}<t}\circ dw_{t_{1}}\otimes \circ dw_{t_{2}}
\end{equation*}
where $\circ d$ denotes the Stratonovich integration. Let $\mathbf{w}%
_{s,t}=(1,w_{s,t}^{1},w_{s,t}^{2})$. The law of $\{\mathbf{w}_{s,t}:(s,t)\in
\Delta \}$ is denoted by $\tilde{P}^{w}$ which is a probability measure on $(%
\mathbb{W}^{p},\mathcal{B}(\mathbb{W}^{p}))$.

\begin{enumerate}
\item  For any $w\in \mathbb{W}$ there is a unique solution $y$ of (\ref
{odef01}) which belongs to $\mathbb{W}$, denoted by $G(w)$. Their
corresponding geometric rough paths are denoted by $\mathbf{w}\in \mathbb{W}%
^{\infty }$ and $G(\mathbf{w})\in \mathbb{W}^{\infty }$. The mapping $G:%
\mathbf{w}\rightarrow F(\mathbf{w})$ can be uniquely extended to be a
continuous mapping from $(\mathbb{W}^{p},d_{p})$ to $(\mathbb{W}^{p},d_{p})$%
, denoted again by $G$, called the It\^{o}-Lyons mapping defined by (\ref
{odef01}). Moreover, the projection to the first level path, $y_{t}=G^{1}(%
\mathbf{w})_{0,t}$ is a version of the strong solution of (\ref{odef01}) on
the probability space $(\mathbb{W}^{p},\mathcal{B}(\mathbb{W}^{p}),\tilde{P}%
^{w})$. The results remain true if all $f_{j}^{i}$ are linear in the space
variables, with bounded derivatives in $t$.

\item  We have 
\begin{equation*}
\tilde{P}^{w}\left\{ \Gamma (\varepsilon )\mathbf{w}\in \mathbb{W}%
^{p}:\forall \varepsilon >0\right\} =1\text{ }
\end{equation*}
where $\Gamma (\varepsilon )\mathbf{w}_{s,t}=(1,\sqrt{\varepsilon }%
w_{s,t}^{1},\varepsilon w_{s,t}^{2})$.

\item  Let $\tilde{P}_{\varepsilon }^{w}$ be the distribution of $\left(
\Gamma (\varepsilon )\mathbf{w}_{s,t}\right) _{0\leq s\leq t\leq 1}$, a
probability measure on $(\mathbb{W}^{p},\mathcal{B}(\mathbb{W}^{p}))$. Then $%
\{\tilde{P}_{\varepsilon }^{w}:\varepsilon >0\}$ possesses the large
deviation principle with respect to the topology induced by the $p$%
-variation metric, with rate function 
\begin{equation*}
\phi (\mathbf{w})=\frac{1}{2}\int_{0}^{1}|\dot{w}(t)|^{2}dt\text{, \ }
\end{equation*}
\ if $\mathbf{w}\in \mathbb{W}^{\infty }$ such that its first level path $%
w\in H_{1}^{1}([0,1];R^{d})$, otherwise $\phi (\mathbf{w})=\infty $.
\end{enumerate}
\end{theorem}

The first item in the theorem is called the universal limit theorem of
Lyons', the second item says the Brownian motion may be lifted to geometric
rough paths, and the last item is Schilder's large deviation principle in
the $p$-variation metric, proved in Ledoux-Qian-Zhang \cite{lqz}.

\subsection{Several elementary facts}

In this part we present some important facts about the relationship between
multiple Wiener-It\^{o} integrals and solutions of stochastic differential
equations of Stratonovich type. To this end we need to introduce more
notations.

If $f\in L^{2}(R_{+}^{n})$, then $J_{n}(f)=\{J_{n}(f)_{t}\}$ is the process
of $n$-th multiple Wiener-It\^{o} integrals where 
\begin{equation*}
J_{n}(f)_{t}=\int_{0<t_{1}<\cdots <t_{n}<t}f(t_{1},\cdots
,t_{n})dw_{t_{1}}\cdots dw_{t_{n}}
\end{equation*}
which is a martingale for $n\geq 1$.

It occurs in the computations below some ``partial'' multiple Wiener-It\^{o}
integrals which are no-longer martingales. Here is a typical example.

If $f$ is a function of $n$-variables $(t_{1},\cdots ,t_{n})$, then for $%
1\leq k\leq n$ we use $f_{;k}(\cdot ;t)$ to denote the function of $%
(t_{1},\cdots ,t_{k})$: 
\begin{equation*}
f_{;k}(\cdot ;t):\left( t_{1},\cdots ,t_{k}\right) \rightarrow
f(t_{1},\cdots ,t_{k},t,\cdots ,t)\text{.}
\end{equation*}
Then $f_{;n}=f$. The following stochastic process 
\begin{equation*}
J_{k}(f_{;k}(\cdot ;t))_{t}=\int_{0<t_{1}<\cdots <t_{k}<t}f(t_{1},\cdots
,t_{k},t,\cdots ,t)dw_{t_{1}}\cdots dw_{t_{n}}
\end{equation*}
is well-defined, for example, if $f$ is differentiable in all variables.

In what follows, we always consider a function $f$ of $n$ variables in the
order from left to right (i.e. we use the standard coordinate system in $%
R^{n}$), and $\nabla _{j}f$ denotes the partial derivative in the $j$-th
coordinate, i.e. $\frac{\partial }{\partial t_{j}}f$ .

\begin{lemma}
If $f(t_{1},\cdots ,t_{n})$ is smooth with bounded derivatives, then 
\begin{equation}
dJ_{n}(f)_{t}=J_{n-1}\left( f_{;n-1}(\cdot ;t)\right) _{t}\circ dw_{t}-\frac{%
1}{2}J_{n-2}(\left( \nabla _{n}f\right) _{;n-2}(\cdot ;t))dt  \label{ew-01}
\end{equation}
where $\circ dw_{t}$ denotes the Stratonovich differential.
\end{lemma}

\begin{proof}
By definition 
\begin{equation*}
J_{n}(f)_{t}=\int_{0}^{t}J_{n-1}\left( f_{;n-1}(\cdot ;s)\right) _{s}dw_{s}%
\text{ .}
\end{equation*}
To simplify our proof, let $Z_{t}=J_{n-1}\left( f_{;n-1}(\cdot ;t)\right)
_{t}$ so that $J_{n}(f)_{t}=\int_{0}^{t}Z_{s}dw_{s}$. Therefore 
\begin{equation*}
J_{n}(f)_{t}=\int_{0}^{t}Z_{s}\circ dw_{s}-\frac{1}{2}\langle Z,w\rangle _{t}
\end{equation*}
and we aim to compute the bracket process $\langle Z,w\rangle _{t}$. To this
end, we begin with the case that 
\begin{equation*}
f_{n}(t_{1},\cdots ,t_{n-1},t_{n})=g_{n-1}(t_{1},\cdots ,t_{n-1})g(t_{n})%
\text{.}
\end{equation*}
Then, according to integration by parts 
\begin{eqnarray*}
Z_{t} &=&g(t)J_{n-1}(g_{n-1})_{t} \\
&=&\int_{0}^{t}g^{\prime }(s)J_{n-1}(g_{n-1})_{s}ds+\int_{0}^{t}g^{\prime
}(s)dJ_{n-1}(g_{n-1})_{s} \\
&=&\int_{0}^{t}g^{\prime }(s)J_{n-1}(g_{n-1})_{s}ds \\
&&+\int_{0}^{t}g^{\prime }(s)J_{n-2}(g_{n-1;n-2}(\cdot ;s))_{s}dw_{s}
\end{eqnarray*}
which follows that 
\begin{eqnarray*}
\langle Z,w\rangle _{t} &=&\int_{0}^{t}g^{\prime
}(s)J_{n-2}(g_{n-1;n-2}(\cdot ;s))_{s}ds \\
&=&\int_{0}^{t}J_{n-2}(g_{n-1;n-2}(\cdot ;s)g^{\prime }(s))_{s}ds \\
&=&\int_{0}^{t}J_{n-2}(\left( \nabla _{n}f\right) _{;n-2}(\cdot ,s))_{s}ds%
\text{.}
\end{eqnarray*}
It is immediate that this equality holds for general $f$, and thus proves
the lemma.
\end{proof}

\begin{lemma}
Let $f_{n}(t_{1},\cdots ,t_{n})$ be a smooth symmetric function, let $1\leq
k\leq n$, and consider It\^{o}'s multiple integral 
\begin{eqnarray*}
H_{t} &=&J_{k}(f_{n;k}(\cdot ;t))_{t} \\
&=&\int_{0<t_{1}<\cdots <t_{k}<t}f_{n}(t_{1},\cdots ,t_{k},t,\cdots
,t)dw_{t_{1}}\cdots dw_{t_{k}}\text{.}
\end{eqnarray*}
Then 
\begin{eqnarray}
dH_{t} &=&\sum_{j=k+1}^{n}J_{k}(\left( \nabla _{j}f_{n}\right) _{;k}(\cdot
;t))_{t}dt+J_{k-1}\left( f_{n;k-1}(\cdot ;t)\right) _{t}\circ dw_{t}  \notag
\\
&&-\frac{1}{2}J_{k-2}(\left( \nabla _{k}f_{n}\right) _{;k-2}(\cdot ;t))dt%
\text{.}  \label{f-r1}
\end{eqnarray}
\end{lemma}

\begin{proof}
Let us consider the case that 
\begin{equation*}
f_{n}(t_{1},\cdots ,t_{k},t_{k+1},\cdots ,t_{n})=g_{k}(t_{1},\cdots
,t_{k})g_{k+1}(t_{k+1})\cdots g_{n}(t_{n})
\end{equation*}
so that 
\begin{equation*}
f_{n}(t_{1},\cdots ,t_{k},t,\cdots ,t)=g_{k}(t_{1},\cdots ,t_{k})g(t)
\end{equation*}
with 
\begin{equation*}
g(t)=g_{k+1}(t)\cdots g_{n}(t)\text{.}
\end{equation*}
Then, by integration by parts, 
\begin{eqnarray*}
dH_{t} &=&g^{\prime }(t)J_{k}(g_{k})_{t}dt+g(t)dJ_{k}(g_{k})_{t} \\
&=&g^{\prime }(t)J_{k}(g_{k})_{t}dt+g(t)J_{k-1}\left( g_{k;k-1}(\cdot
;t)\right) _{t}\circ dw_{t} \\
&&-\frac{1}{2}g(t)J_{k-2}(\left( \nabla _{k}g_{k}\right) _{;k-2}(\cdot ;t))dt
\\
&=&\sum_{j=k+1}^{n}J_{k}(\left( \nabla _{j}f_{n}\right) _{;k}(\cdot
;t))_{t}dt+J_{k-1}\left( f_{n;k-1}(\cdot ;t)\right) _{t}\circ dw_{t} \\
&&-\frac{1}{2}J_{k-2}(\left( \nabla _{k}f_{n}\right) _{;k-2}(\cdot ;t))dt
\end{eqnarray*}
which proves the lemma.
\end{proof}

\subsection{Stochastic differential equations}

Let $\xi =\sum_{n=1}^{N}J_{n}(f_{n})_{1}\in L^{2}(\mathbf{W}_{0}^{d},%
\mathcal{F}_{1},P^{w})$ (but in the following computations, we assume for
simplicity that $d=1$) for some natural number $N$ and smooth functions $%
f_{n}$ on $[0,1]^{n}$ with bounded derivatives, and 
\begin{equation}
Y_{t}^{\varepsilon }=P_{-\log \sqrt{\varepsilon }}Y_{t}=\sum_{n=1}^{N}%
\varepsilon ^{\frac{n}{2}}J_{n}(f_{n})_{t}\text{, \ \ \ \ }t\in \lbrack 0,1]%
\text{.}  \label{NY1}
\end{equation}

The aim of this section is to construct a continuous function $%
F^{\varepsilon }$ on $(\mathbb{W}^{p},\mathcal{B}(\mathbb{W}^{p}),\tilde{P}%
^{w})$ (where the space $\mathbb{W}^{p}$ of geometric rough paths is endowed
with the $p$-variation distance), such that $F^{\varepsilon }(\Gamma
(\varepsilon )\mathbf{w})=Y^{\varepsilon }(w)$ almost surely.

To this end, we demonstrate that $Y^{\varepsilon }$ is a part of the
solution of a Stratonovich type stochastic differential equation, at least
for good functions $f_{n}$.

According to (\ref{f-r1}) 
\begin{eqnarray*}
dY_{t}^{\varepsilon } &=&\sqrt{\varepsilon }\sum_{n=1}^{N}\varepsilon ^{%
\frac{n-1}{2}}J_{n-1}\left( f_{n;n-1}(\cdot ;t)\right) _{t}\circ dw_{t} \\
&&-\frac{1}{2}\varepsilon \sum_{n=1}^{N}\varepsilon ^{\frac{n-2}{2}%
}J_{n-2}(\left( \nabla _{n}f_{n}\right) _{;n-2}(\cdot ;t))dt\text{,}
\end{eqnarray*}
and 
\begin{eqnarray}
dJ_{n-1}\left( f_{n;n-1}(\cdot ;t)\right) _{t} &=&\nabla _{n}\left[
J_{n-1}(f_{n;n-1}(\cdot ;t))_{t}\right] dt  \notag \\
&&+J_{n-2}\left( f_{n;n-2}(\cdot ;t)\right) _{t}\circ dw_{t}  \notag \\
&&-\frac{1}{2}J_{n-3}(\left( \nabla _{n-1}f_{n}\right) _{;n-3}(\cdot ;t))dt%
\text{.}  \label{a1}
\end{eqnarray}
where 
\begin{eqnarray*}
&&\nabla _{n}\left[ J_{n-1}(f_{n;n-1}(\cdot ;t))_{t}\right] \\
&=&\left. \frac{\partial }{\partial t_{n}}\right|
_{t_{n}=t}\int_{0<t_{1}<\cdots <t_{n-1}<t}f_{n}(t_{1},\cdots
,t_{n-1},t_{n})dw_{t_{1}}\cdots dw_{t_{n-1}}\text{.}
\end{eqnarray*}

Unfortunately it does not lead to a closed system of stochastic differential
equations of Stratonovich type. Therefore we consider a special case in
which each $f_{n}$ is a linear combination of functions of product form. We
need some more notations.

For $\varepsilon \in (0,1)$, $n\in \mathbb{N}$, and $\{g;f^{1},\cdots
,f^{n}\}$ a family of smooth functions on $[0,1]$ with bounded derivatives,
then we define 
\begin{equation}
Z_{t}^{n,\{g;f^{1},\cdots ,f^{n}\}}=\varepsilon ^{\frac{n}{2}%
}g(t)\int_{0<t_{1}<\cdots <t_{n}<t}f^{1}(t_{1})\cdots
f^{n}(t_{n})dw_{t_{1}}\cdots dw_{t_{n}}  \label{df1}
\end{equation}
and 
\begin{equation*}
X_{t}^{n,\{g;f^{1},\cdots ,f^{n}\}}=\left( 
\begin{array}{c}
Z_{t}^{n,\{1;f^{1},\cdots ,f^{n}\}} \\ 
Z_{t}^{n,\{g;f^{1},\cdots ,f^{n}\}}
\end{array}
\right) \text{.}
\end{equation*}
Therefore 
\begin{eqnarray*}
Z_{t}^{n,\{1;f^{1},\cdots ,f^{n}\}} &=&\varepsilon ^{\frac{n}{2}%
}\int_{0<t_{1}<\cdots <t_{n}<t}f^{1}(t_{1})\cdots
f^{n}(t_{n})dw_{t_{1}}\cdots dw_{t_{n}}\text{,} \\
Z_{t}^{n,\{g;f^{1},\cdots ,f^{n}\}} &=&g(t)Z_{t}^{n,\{1;f^{1},\cdots
,f^{n}\}}
\end{eqnarray*}
and $X_{t}^{n,\{1;f^{1},\cdots ,f^{n}\}}$ contains just two identical copies
of $Z_{t}^{n,\{1;f^{1},\cdots ,f^{n}\}}$. We use the convention that $%
Z^{0,\{\cdots \}}=1$ and $Z^{n,\{\cdots \}}=0$ for $n<0$.

\begin{lemma}
\label{lem-0a1}The stochastic process $X_{t}^{n,\{g;f^{1},\cdots ,f^{n}\}}$
satisfies the following recursion equations 
\begin{eqnarray}
dX_{t}^{n,\{g;f^{1},\cdots ,f^{n}\}} &=&\varepsilon ^{\frac{n}{2}}g^{\prime
}(t)E_{21}X_{t}^{n,\{g;f^{1},\cdots ,f^{n}\}}dt-\frac{1}{2}\varepsilon
E_{12}X_{t}^{n-2,\{f^{n-1}f^{n};f^{1},\cdots ,f^{n-2}\}}dt  \notag \\
&&-\frac{1}{2}\varepsilon E_{22}X_{t}^{n-2,\{gf^{n-1}f^{n};f^{1},\cdots
,f^{n-2}\}}dt  \label{eq3} \\
&&+\sqrt{\varepsilon }\left( E_{12}X_{t}^{n-1,\{f^{n};f^{1},\cdots
,f^{n-1}\}}+E_{22}X_{t}^{n-1,\{gf^{n};f^{1},\cdots ,f^{n-1}\}}\right) \circ
dw_{t}\text{,}  \notag
\end{eqnarray}
where $E_{ij}$ is the $2\times 2$ matrices with $1$ at $(i,j)$ entry and
other entries zero.
\end{lemma}

\begin{proof}
It follows from (\ref{f-r1}) that $Z_{t}^{n,\{g;f^{1},\cdots ,f^{n}\}}$
satisfies the following stochastic differential equation 
\begin{eqnarray*}
dZ_{t}^{n,\{g;f^{1},\cdots ,f^{n}\}} &=&\sqrt{\varepsilon }%
Z_{t}^{n-1,\{gf^{n};f^{1},\cdots ,f^{n-1}\}}\circ dw_{t}+\varepsilon ^{\frac{%
n}{2}}g^{\prime }(t)Z_{t}^{n,\{1;f^{1},\cdots ,f^{n}\}}dt \\
&&-\frac{1}{2}\varepsilon Z_{t}^{n-2,\{gf^{n-1}f^{n};f^{1},\cdots
,f^{n-2}\}}dt
\end{eqnarray*}
and 
\begin{equation*}
dZ_{t}^{n,\{1;f^{1},\cdots ,f^{n}\}}=\sqrt{\varepsilon }Z_{t}^{n-1,%
\{f^{n};f^{1},\cdots ,f^{n-1}\}}\circ dw_{t}-\frac{1}{2}\varepsilon
Z_{t}^{n-2,\{f^{n-1}f^{n};f^{1},\cdots ,f^{n-2}\}}dt
\end{equation*}
which is equivalent to (\ref{eq3}).
\end{proof}

Now let us consider 
\begin{equation*}
Y_{t}^{\varepsilon }=P_{-\log \sqrt{\varepsilon }}P^{w}(\xi |\mathcal{F}_{t})
\end{equation*}
where $\xi =\sum_{n=1}^{N}J_{n}(f_{n})_{1}\in L^{2}(\mathbf{W}_{0}^{1},%
\mathcal{F}_{1},P^{w})$ with each integrand $f_{n}(t_{1},\cdots ,t_{n})$ has
a product form, say 
\begin{equation}
f_{n}(t_{1},\cdots ,t_{n})=\sum_{j_{1},\cdots
,j_{n}=1}^{N_{n}}C_{n}^{j_{1}\cdots j_{n}}f_{n}^{j_{1}}(t_{1})\cdots
f_{n}^{j_{n}}(t_{n})  \label{ej-1}
\end{equation}
where $C_{n}^{j_{1}\cdots j_{n}}$ are constants, $N_{n}$ is a natural
number, and all $f_{k}^{j_{i}}$ are smooth functions with bounded
derivatives.

In this case 
\begin{equation}
Y_{t}^{\varepsilon }=\sum_{n=1}^{N}\sum_{j_{1},\cdots
,j_{n}=1}^{N_{n}}C_{n}^{j_{1}\cdots j_{n}}\varepsilon ^{\frac{n}{2}%
}\int_{0<t_{1}<\cdots <t_{n}<t}f_{n}^{j_{1}}(t_{1})\cdots
f_{n}^{j_{n}}(t_{n})dw_{t_{1}}\cdots dw_{t_{n}}\text{.}  \label{c02}
\end{equation}
We are going to show that $Y_{t}^{\varepsilon }$ is part of the solution to
a Stratonovich type stochastic differential equation. More precisely, we are
going to show that 
\begin{equation}
\left( Y_{t}^{\varepsilon },Y_{t}^{\varepsilon },(X_{t}^{n,k,j})_{0\leq
j<k<n,1\leq n\leq N}\right)  \label{c-0a1}
\end{equation}
is the unique strong solution to a system of stochastic differential
equations of Stratonovich type, where the general term is given by 
\begin{equation*}
X_{t}^{n,k,j}\equiv X_{t}^{n-k,\left\{ g_{n,k,j};f_{n}^{j_{1}},\cdots
,f_{n}^{j_{n-k}}\right\} }\text{ \ \ \ for }k=1,\cdots ,n-1,j=0,\cdots ,k-1
\end{equation*}
and $g_{n,k,j}=\prod_{i=j}^{k-1}f_{n}^{j_{n-i}}$. The projection to the
first component in (\ref{c-0a1}), i.e. 
\begin{equation*}
\left( y,y,(x^{n,k,j})_{0\leq j<k<n,1\leq n\leq N}\right) \rightarrow y
\end{equation*}
will be denoted by $\pi _{1}$.

\begin{proposition}
\label{prop0.1}Let $Z_{t}=\left( Y_{t}^{\varepsilon },Y_{t}^{\varepsilon
}\right) $ where $(Y_{t}^{\varepsilon })_{t\leq 1}$ be given by equation (%
\ref{c02}), and let 
\begin{equation*}
X_{t}=\left( Z_{t},\left( X_{t}^{n,k,j}\right) _{0\leq j<k<n,1\leq n\leq
N}\right)
\end{equation*}
Then the stochastic process $X_{t}$ is the unique strong solution to the
following system of Stratonovich stochastic differential equations 
\begin{eqnarray}
dZ_{t} &=&\sqrt{\varepsilon }\sum_{n=1}^{N}\sum_{j_{1},\cdots ,j_{n}=1}^{%
\tilde{N}}C_{n}^{j_{1}\cdots j_{n}}\left( E_{12}+E_{22}\right)
X_{t}^{n,1,0}\circ dw_{t}  \notag \\
&&-\frac{1}{2}\varepsilon \sum_{n=1}^{N}\sum_{j_{1},\cdots ,j_{n}=1}^{\tilde{%
N}}C_{n}^{j_{1}\cdots j_{n}}\left( E_{12}+E_{22}\right) X_{t}^{n,2,0}dt\text{%
,}  \label{sz1}
\end{eqnarray}
\begin{eqnarray}
dX_{t}^{n,k,j} &=&\varepsilon ^{\frac{n-k}{2}}g_{n,j}^{\prime
}(t)E_{21}X_{t}^{n,k,j}dt  \notag \\
&&+\sqrt{\varepsilon }\left(
E_{12}X_{t}^{n,k+1,k}+E_{22}X_{t}^{n,k+1,j}\right) \circ dw_{t}  \label{sz2}
\\
&&-\frac{1}{2}\varepsilon \left(
E_{12}X_{t}^{n,k+2,k}+E_{22}X_{t}^{n,k+2,j}\right) dt  \notag
\end{eqnarray}
for $0\leq j<k<n\leq N$, $X_{t}^{n,n,j}=1$ and $X_{t}^{n,k,j}=0$ for any $%
k>n $. The system (\ref{sz1},\ref{sz2}) can be written into a compact form 
\begin{eqnarray}
dX_{t} &=&\sqrt{\varepsilon }A(X_{t})\circ dw_{t}+\varepsilon B(X_{t})dt 
\notag \\
&&+\sum_{k=1}^{N-1}\varepsilon ^{\frac{k}{2}}C_{k}(t,X_{t})dt  \label{ef02}
\end{eqnarray}
where all $A$, $B$ and $C_{k}$ defined by (\ref{sz1},\ref{sz2}) are linear
in the space variable, with bounded derivatives in $t$. $\circ dw_{t}$
denotes the Stratonovich differential.
\end{proposition}

\begin{proof}
By definition 
\begin{equation}
Y_{t}^{\varepsilon }=\sum_{n=1}^{N}\sum_{j_{1},\cdots ,j_{n}=1}^{\tilde{N}%
}C_{n}^{j_{1}\cdots j_{n}}Z_{t}^{n,\{1;f_{n}^{j_{1}},\cdots ,f_{n}^{j_{n}}\}}%
\text{.}  \label{ef4}
\end{equation}
Instead of considering $Y^{\varepsilon }$ we take two copies of the same
equation, i.e. we consider 
\begin{eqnarray*}
Z_{t}^{\varepsilon } &=&\left( 
\begin{array}{c}
Y_{t}^{\varepsilon } \\ 
Y_{t}^{\varepsilon }
\end{array}
\right) \\
&=&\sum_{n=1}^{N}\sum_{j_{1},\cdots ,j_{n}=1}^{\tilde{N}}C_{n}^{j_{1}\cdots
j_{n}}X_{t}^{n,\{1;f_{n}^{j_{1}},\cdots ,f_{n}^{j_{n}}\}}
\end{eqnarray*}
so that 
\begin{equation}
dZ_{t}^{\varepsilon }=\sum_{n=1}^{N}\sum_{j_{1},\cdots ,j_{n}=1}^{\tilde{N}%
}C_{n}^{j_{1}\cdots j_{n}}dX_{t}^{n,\{1;f_{n}^{j_{1}},\cdots
,f_{n}^{j_{n}}\}}\text{ .}  \label{ef-0a1}
\end{equation}
Using (\ref{eq3}) we obtain 
\begin{eqnarray}
dX_{t}^{n,\{1;f_{n}^{j_{1}},\cdots ,f_{n}^{j_{n}}\}} &=&\sqrt{\varepsilon }%
\left( E_{12}+E_{22}\right) X_{t}^{n-1,\{f_{n}^{j_{n}};f_{n}^{j_{1}},\cdots
,f_{n}^{j_{n-1}}\}}\circ dw_{t}  \notag \\
&&-\frac{1}{2}\varepsilon \left( E_{12}+E_{22}\right)
X_{t}^{n-2,\{f_{n}^{j_{n}}f_{n}^{j_{n-1}};f_{n}^{j_{1}},\cdots
,f_{n}^{j_{n-2}}\}}dt  \label{s-1}
\end{eqnarray}
so that 
\begin{eqnarray}
dZ_{t}^{\varepsilon } &=&\sqrt{\varepsilon }\sum_{n=1}^{N}\sum_{j_{1},\cdots
,j_{n}=1}^{\tilde{N}}C_{n}^{j_{1}\cdots j_{n}}\left( E_{12}+E_{22}\right)
X_{t}^{n-1,\{f_{n}^{j_{n}};f_{n}^{j_{1}},\cdots ,f_{n}^{j_{n-1}}\}}\circ
dw_{t}  \notag \\
&&-\frac{1}{2}\varepsilon \sum_{n=1}^{N}\sum_{j_{1},\cdots ,j_{n}=1}^{\tilde{%
N}}C_{n}^{j_{1}\cdots j_{n}}\left( E_{12}+E_{22}\right)
X_{t}^{n-2,\{f_{n}^{j_{n}}f_{n}^{j_{n-1}};f_{n}^{j_{1}},\cdots
,f_{n}^{j_{n-2}}\}}dt\text{.}  \label{s-2}
\end{eqnarray}
Now repeating the use of Lemma \ref{lem-0a1} we obtain 
\begin{eqnarray*}
&&dX_{t}^{n-1,\{f_{n}^{j_{n}};f_{n}^{j_{1}},\cdots ,f_{n}^{j_{n-1}}\}} \\
&=&\varepsilon ^{\frac{n-1}{2}}\frac{df_{n}^{j_{n}}}{dt}E_{21}X_{t}^{n-1,%
\{f_{n}^{j_{n}};f_{n}^{j_{1}},\cdots ,f_{n}^{j_{n-1}}\}}dt-\frac{1}{2}%
\varepsilon
E_{12}X_{t}^{n-3,\{f_{n}^{j_{n-2}}f_{n}^{j_{n-1}};f_{n}^{j_{1}},\cdots
,f_{n}^{j_{n-3}}\}}dt \\
&&-\frac{1}{2}\varepsilon
E_{22}X_{t}^{n-3,%
\{f_{n}^{j_{n-2}}f_{n}^{j_{n-1}}f_{n}^{j_{n}};f_{n}^{j_{1}},\cdots
,f_{n}^{j_{n-3}}\}}dt \\
&&+\sqrt{\varepsilon }\left(
E_{12}X_{t}^{n-2,\{f_{n}^{j_{n-1}};f_{n}^{j_{1}},\cdots
,f_{n}^{j_{n-2}}\}}+E_{22}X_{t}^{n-2,%
\{f_{n}^{j_{n}}f_{n}^{j_{n-1}};f_{n}^{j_{1}},\cdots
,f_{n}^{j_{n-2}}\}}\right) \circ dw_{t}\text{,}
\end{eqnarray*}
\begin{eqnarray*}
&&dX_{t}^{n-2,\{f_{n}^{j_{n-1}};f_{n}^{j_{1}},\cdots ,f_{n}^{j_{n-2}}\}} \\
&=&\varepsilon ^{\frac{n-2}{2}}\frac{df_{n}^{j_{n-1}}}{dt}%
E_{21}X_{t}^{n-2,\{f_{n}^{j_{n-1}};f_{n}^{j_{1}},\cdots ,f_{n}^{j_{n-2}}\}}dt
\\
&&+\sqrt{\varepsilon }\left(
E_{12}X_{t}^{n-3,\{f_{n}^{j_{n-2}};f_{n}^{j_{1}},\cdots
,f_{n}^{j_{n-3}}\}}+E_{22}X_{t}^{n-3,%
\{f_{n}^{j_{n-1}}f_{n}^{j_{n-2}};f_{n}^{j_{1}},\cdots
,f_{n}^{j_{n-3}}\}}\right) \circ dw_{t} \\
&&-\frac{1}{2}\varepsilon
E_{12}X_{t}^{n-4,\{f_{n}^{j_{n-3}}f_{n}^{j_{n-2}};f_{n}^{j_{1}},\cdots
,f_{n}^{j_{n-4}}\}}dt \\
&&-\frac{1}{2}\varepsilon
E_{22}X_{t}^{n-4,%
\{f_{n}^{j_{n-3}}f_{n}^{j_{n-2}}f_{n}^{j_{n-1}};f_{n}^{j_{1}},\cdots
,f_{n}^{j_{n-4}}\}}dt\text{,}
\end{eqnarray*}
and 
\begin{eqnarray*}
&&dX_{t}^{n-2,\{f_{n}^{j_{n}}f_{n}^{j_{n-1}};f_{n}^{j_{1}},\cdots
,f_{n}^{j_{n-2}}\}} \\
&=&\varepsilon ^{\frac{n-2}{2}}\frac{d\left(
f_{n}^{j_{n}}f_{n}^{j_{n-1}}\right) }{dt}E_{21}X_{t}^{n-2,%
\{f_{n}^{j_{n}}f_{n}^{j_{n-1}};f_{n}^{j_{1}},\cdots ,f_{n}^{j_{n-2}}\}}dt \\
&&+\sqrt{\varepsilon }\left(
E_{12}X_{t}^{n-3,\{f_{n}^{j_{n-2}};f_{n}^{j_{1}},\cdots
,f_{n}^{j_{n-3}}\}}+E_{22}X_{t}^{n-3,%
\{f_{n}^{j_{n}}f_{n}^{j_{n-1}}f_{n}^{j_{n-2}};f_{n}^{j_{1}},\cdots
,f_{n}^{j_{n-3}}\}}\right) \circ dw_{t} \\
&&-\frac{1}{2}\varepsilon
E_{12}X_{t}^{n-4,\{f_{n}^{j_{n-3}}f_{n}^{j_{n-2}};f_{n}^{j_{1}},\cdots
,f_{n}^{j_{n-4}}\}}dt \\
&&-\frac{1}{2}\varepsilon
E_{22}X_{t}^{n-4,%
\{f_{n}^{j_{n-3}}f_{n}^{j_{n-2}}f_{n}^{j_{n-1}}f_{n}^{j_{n}};f_{n}^{j_{1}},%
\cdots ,f_{n}^{j_{n-4}}\}}dt
\end{eqnarray*}
and so on. The general term appearing in this system is 
\begin{equation*}
X_{t}^{n,k,j}\equiv X_{t}^{n-k,\left\{ g_{n,k,j};f_{n}^{j_{1}},\cdots
,f_{n}^{j_{n-k}}\right\} }\text{, \ \ \ }k=1,\cdots ,n-1,j=0,\cdots ,k-1%
\text{,}
\end{equation*}
where $g_{n,k,j}=\prod_{i=j}^{k-1}f_{n}^{j_{n-i}}$, $n$ runs through $1$ up
to $N$. We have thus completed the proof.
\end{proof}

\subsection{It\^{o}-Lyons mappings}

Let $\xi =\sum_{n=1}^{N}J_{n}(f_{n})_{1}$ be given by (\ref{ej-1}) and use
the notations in the previous sub-section. For each $\delta \in (0,1)$, we
consider the following differential equation (\ref{ef02}) 
\begin{eqnarray}
dX_{t} &=&A(X_{t})\circ dw_{t}+\delta B(X_{t})dt  \notag \\
&&+\sum_{k=1}^{N-1}\delta ^{\frac{k}{2}}C_{k}(t,X_{t})dt  \label{r-e1}
\end{eqnarray}
on the rough path space $\mathbb{W}^{p}$, where $A,B$ and $C_{k}$ are given
in Proposition \ref{prop0.1}. According to Theorem \ref{hgj01}, the
differential equation (\ref{r-e1}) defines an It\^{o}-Lyons mapping $%
G^{\delta }$ which is continuous with respect to the $p$-variation topology.
The projection of $G^{\delta }$ to the first component $Y$ in Proposition 
\ref{prop0.1} of the first level path (the projection is denoted by $\pi
_{1} $) is then denoted by $F^{\delta }$. That is $F^{\delta }(\mathbf{w}%
)_{t}=\pi _{1}\left( G^{\delta }\left( \mathbf{w}\right) _{0,t}^{1}\right) $%
. We also consider the differential equation 
\begin{equation*}
dX_{t}=A(X_{t})\circ dw_{t}\text{, \ }X_{0}=0
\end{equation*}
whose corresponding It\^{o}-Lyons mappings are denoted by $\tilde{G}$ and $%
\tilde{F}$ (i.e. $\tilde{G}=G^{0}$ and $\tilde{F}=F^{0}$).

Let us list some properties about $F^{\delta }$.

Recall that $\mathbb{W}^{p}$ is the space of all rough paths in\thinspace $%
R^{d}$ endowed with the $p$-variation metric, $\tilde{\mu}$ is the
distribution of Brownian motion with its area process, $(\mathbf{W}_{0}^{d},%
\mathcal{F}_{1},P^{w})$ is the Wiener space, and $\mathbf{W}_{0}^{d}$
equipped with the uniform norm. The natural projection $\pi :\mathbb{W}%
^{p}\rightarrow \mathbf{W}_{0}^{d}$ which takes $\mathbf{w}%
=(1,w_{s,t}^{1},w_{s,t}^{2})$ to its first level path $w:t\in \lbrack
0,1]\rightarrow w_{0,t}^{1}$ is continuous.

\begin{proposition}
\label{prop-m}1) For each $\delta \in (0,1)$, $F^{\delta }:\mathbb{W}%
^{p}\rightarrow \mathbf{W}_{0}^{d}$ is continuous, where $\mathbb{W}^{p}$ is
equipped with the $p$-variation metric, $\mathbf{W}_{0}^{d}$ endowed with
the uniform norm.

2) If $h\in H_{0}^{1}([0,1];R^{d})$, then $F^{\varepsilon }\left( \Gamma
(\varepsilon )\mathbf{h}\right) _{t}=\pi _{1}(x_{t})$ where $(x_{t})$ is the
unique solution to the ordinary differential equation 
\begin{equation}
dx_{t}=\sqrt{\varepsilon }A(x_{t})dh_{t}+\varepsilon
B(x_{t})dt+\sum_{k=1}^{N-1}\varepsilon ^{\frac{k}{2}}C_{k}(t,x_{t})dt\text{.}
\label{r-e2}
\end{equation}

3) For every $\varepsilon \in (0,1)$, we have 
\begin{equation}
\tilde{P}^{w}\left\{ \mathbf{w}:F^{\varepsilon }\left( \Gamma (\varepsilon )%
\mathbf{w}\right) _{t}=Y_{t}^{\varepsilon }(w)\text{ \ }\forall t\in \lbrack
0,1]\right\} =1\text{.}  \label{r-e3}
\end{equation}
where $w=w_{0,t}^{1}$ is the first level path of $\mathbf{w}%
=(1,w^{1},w^{2})\in \mathbb{W}^{p}$.

4) We have 
\begin{equation}
\tilde{P}^{w}\left\{ \mathbf{w}:\tilde{F}\left( \Gamma (\varepsilon )\mathbf{%
w}\right) _{t}=S_{t}^{\varepsilon }(w)\text{ \ }\forall t\in \lbrack
0,1]\right\} =1\text{.}  \label{r-e4}
\end{equation}
where 
\begin{equation*}
S_{t}^{\varepsilon }=\sum_{n=1}^{N}\varepsilon ^{\frac{n}{2}%
}\int_{0}^{t}f_{n}(t_{1},\cdots ,t_{n})\circ dw_{t_{1}}\cdots \circ
dw_{t_{n}}
\end{equation*}
and, if $\mathbf{h}\in \mathbb{W}^{\infty }$ such that $t\rightarrow
h_{t}=h_{0,t}^{1}$ $\in H_{0}^{1}([0,1];R^{d})$, then 
\begin{equation*}
\tilde{F}\left( \mathbf{h}\right)
=\sum_{n=1}^{N}\int_{0}^{t}f_{n}(t_{1},\cdots ,t_{n})\dot{h}(t_{1})\cdots 
\dot{h}(t_{n})dt_{1}\cdots dt_{n}\text{.}
\end{equation*}
\end{proposition}

\begin{proof}
The first claim and second claim follow from Lyons' continuity theorem,
Theorem \ref{hgj01}. 3) follows from Proposition \ref{prop0.1} and Theorem 
\ref{hgj01}. The last item comes from the fact that the terms involving
vector fields $B$ and $C_{k}$ come from the correction terms from Ito
integrals to Stratonovich's integrals, therefor if we started with the
multiple Stratonovich's integrals (or ordinary integrals), all these terms
disappeared. We thus completed the proof.
\end{proof}

\begin{proposition}
\label{lem-b1} Let $\delta >0$. Consider the solutions $(x_{t}^{\varepsilon
})_{t\geq 0}$ and $(y_{t}^{\varepsilon })_{t\geq 0}$ be the solutions to
Stratonovich differential equations 
\begin{eqnarray*}
dy_{t} &=&\sqrt{\varepsilon }A(y_{t})\circ dw_{t}+\varepsilon B(y_{t})dt \\
&&+\sum_{k=1}^{N-1}\varepsilon ^{\frac{k}{2}}C_{k}(t,y_{t})dt\text{ , \ \ }%
y_{0}=0
\end{eqnarray*}
and 
\begin{equation*}
dx_{t}=\sqrt{\varepsilon }A(x_{t})\circ dw_{t}\text{, \ \ }x_{0}=0
\end{equation*}
on $(\mathbf{W}_{0}^{d},\mathcal{F}_{1},P^{w})$, respectively, where $A,B$
and $C_{k}$ are given in Proposition \ref{prop-m}. Then 
\begin{equation*}
\lim_{\varepsilon \rightarrow 0}\varepsilon \log P^{w}\left\{ \sup_{t\in
\lbrack 0,1]}|\pi _{1}(x_{t}^{\varepsilon })-\pi _{1}(y_{t}^{\varepsilon
})|>\delta \right\} =-\infty \text{.}
\end{equation*}
\end{proposition}

\begin{proof}
According to the definition our system 
\begin{equation*}
\pi _{1}(y_{t}^{\varepsilon })=\sum_{n=1}^{N}\varepsilon ^{\frac{n}{2}%
}\int_{0}^{t}f_{n}(t_{1},\cdots ,t_{n})dw_{t_{1}}\cdots dw_{t_{n}}
\end{equation*}
and 
\begin{equation*}
\pi _{1}(x_{t}^{\varepsilon })=\sum_{n=1}^{N}\varepsilon ^{\frac{n}{2}%
}\int_{0}^{t}f_{n}(t_{1},\cdots ,t_{n})\circ dw_{t_{1}}\cdots \circ
dw_{t_{n}}
\end{equation*}
so that (for example, by applying Hu-Meyer formula \cite{Hu-Meyer}, see also 
\cite{Mayer-Nua-Pere}) 
\begin{equation*}
\lim_{\varepsilon \rightarrow 0}\varepsilon \log P^{w}\left\{ \sup_{t\in
\lbrack 0,1]}|\pi _{1}(x_{t}^{\varepsilon })-\pi _{1}(y_{t}^{\varepsilon
})|>\delta \right\} =-\infty \text{.}
\end{equation*}
\end{proof}

\begin{corollary}
\label{coro0z}Let $\nu _{\varepsilon }$ be the law of $(\pi
_{1}(y_{t}^{\varepsilon }))_{t\in \lbrack 0,1]}$. Then the family $\{\nu
_{\varepsilon }:\varepsilon \in (0,1)\}$ satisfies the large deviation
principle with the rate function given by 
\begin{equation}
I_{N}^{\prime }(w)=\inf \left\{ I(h):h\in H\text{ s.t. }\Phi (h)=w\right\}
\label{r1}
\end{equation}
where $I(h)=\frac{1}{2}\int_{0}^{1}|\dot{h}(t)|^{2}dt$ for $h\in H$, and 
\begin{equation}
\Phi (h)_{t}=\sum_{n=1}^{N}\int_{0}^{t}f_{n}(t_{1},\cdots ,t_{n})\dot{h}%
(t_{1})\cdots \dot{h}(t_{n})dt_{1}\cdots dt_{n}\text{.}  \label{r2}
\end{equation}
\end{corollary}

\begin{proof}
Let $\tilde{G}$ be the It\^{o}-Lyons mapping determined by the differential
equation 
\begin{equation*}
dx_{t}=A(x_{t})\circ dw_{t}\text{, \ \ }x_{0}=0
\end{equation*}
on $\mathbb{W}^{p}$. Then $\tilde{G}:\mathbb{W}^{p}\rightarrow \mathbb{W}%
^{p} $ is continuous. Define $\tilde{F}:\mathbb{W}^{p}\rightarrow \mathbf{W}%
_{0}^{d}$ by 
\begin{equation*}
\tilde{F}(\mathbf{w})_{t}=\pi _{1}\left( \tilde{G}(\mathbf{w)}%
_{0,t}^{1}\right)
\end{equation*}
which is continuous, and moreover $\tilde{F}(\Gamma (\varepsilon )\mathbf{w}%
)_{t}=\pi _{1}(x_{t}^{\varepsilon })$, and 
\begin{equation*}
\tilde{F}(\mathbf{h})_{t}=\sum_{n=1}^{N}\int_{0}^{t}f_{n}(t_{1},\cdots
,t_{n})\dot{h}(t_{1})\cdots \dot{h}(t_{n})dt_{1}\cdots dt_{n}
\end{equation*}
for any $\mathbf{h}\in \mathbb{W}^{\infty }$ such that $h=\pi _{1}(\mathbf{h}%
)\in H_{0}^{1}([0,1];R^{d})$. It follows from Theorem \ref{hgj01}, the
distributions of $(\pi _{1}(x_{t}^{\varepsilon }))$ satisfy the large
deviation principle with the good rate function $I_{N}^{\prime }$ defined by
(\ref{r1}). Now, according to Theorem 4.2.13 on page 130, \cite{LD
techniques by Dembo and Zeitouni} and Proposition \ref{lem-b1}, one may
conclude that $\{\nu _{\varepsilon }:\varepsilon \in (0,1)\}$ satisfies the
large deviation principle with the same rate function. The proof is complete.
\end{proof}

\section{Large deviations for martingales}

In this section we extend the large deviation principle to a general
square-integrable martingale on $(\mathbf{W}_{0}^{d},\mathcal{F}_{1},%
\mathcal{F}_{t},P^{w})$.

Let $\xi \in L^{2}(\mathbf{W}_{0}^{d},\mathcal{F}_{1},P^{w})$ with mean
zero, whose Wiener-It\^{o} chaos decomposition $\xi =\sum_{k=1}^{\infty
}J_{k}(f_{k})_{1}$,where $f_{k}\in L^{2}[0,1]^{k}$ for every $k$ and 
\begin{equation*}
||\xi ||_{2}^{2}=\sum_{n=1}^{\infty }\frac{1}{k!}%
||f_{k}||_{L^{2}[0,1]^{k}}^{2}\text{.}
\end{equation*}
Let $Y_{t}=P^{w}(\xi |\mathcal{F}_{t})=\sum_{k=1}^{\infty }J_{k}(f_{k})_{t}$%
, and for each $\varepsilon \in (0,1)$%
\begin{equation*}
Y_{t}^{\varepsilon }=P_{-\log \sqrt{\varepsilon }}Y_{t}=\sum_{k=1}^{\infty
}\varepsilon ^{\frac{k}{2}}J_{k}(f_{k})_{t}\text{ \ \ \ \ }\forall t\in
\lbrack 0,1]\text{.}
\end{equation*}
Let $\nu _{\varepsilon }$ be the law of $(Y_{t}^{\varepsilon })_{t\in
\lbrack 0,1]}$ which is a probability measure on $(\mathbf{W}_{0}^{d},%
\mathcal{F}_{1})$. In this section, we prove the main result, Theorem \ref
{LD for Martingale}, that is, we show that $\{\nu _{\varepsilon
}:\varepsilon \in (0,1)\}$ satisfies the large deviation principle on $(%
\mathbf{W}_{0}^{d},||\cdot ||)$.

The idea, as we have mentioned, is to construct a sequence of exponential
approximations to $(Y_{t}^{\varepsilon })_{t\in \lbrack 0,1]}$. For each
natural number $n$, there is a natural number $N_{n}$ such that 
\begin{equation*}
\sum_{k=N_{n}+1}^{\infty }\frac{1}{k!}||f_{k}||_{L^{2}[0,1]^{k}}^{2}<\frac{1%
}{2n^{2}}\text{.}
\end{equation*}
For each $k=1,\cdots ,N_{n}$, choose a symmetric function $f_{k}^{n}$ on $%
[0,1]^{k}$ which has a product form 
\begin{equation*}
f_{k}^{n}(t_{1},\cdots ,t_{k})=\sum_{j_{1},\cdots
,j_{n}=1}^{N_{n,k}}C_{n,k}^{j_{1}\cdots j_{k}}f_{k}^{j_{1}}(t_{1})\cdots
f_{k}^{j_{n}}(t_{k})
\end{equation*}
where $C_{n,k}^{j_{1}\cdots j_{k}}$ are constants, $N_{n,k}$ is a and all $%
f_{k}^{j_{i}}$ are smooth functions on $[0,1]$ with bounded derivatives,
such that 
\begin{equation*}
||f_{k}-f_{k}^{n}||_{L^{2}[0,1]^{k}}^{2}<\frac{1}{2en^{2}}\text{ \ \ for }%
k=1,\cdots ,N_{n}\text{.}
\end{equation*}
Define $\xi _{n}=\sum_{k=1}^{N_{n}}I_{k}(f_{k}^{n})_{1}$. Then 
\begin{equation*}
\sum_{k=1}^{N_{n}}\frac{1}{k!}||f_{k}-f_{k}^{n}||_{L^{2}[0,1]^{k}}^{2}<\frac{%
1}{2n^{2}}\text{.}
\end{equation*}
so that 
\begin{eqnarray*}
||\xi -\xi _{n}||_{2}^{2} &=&\sum_{k=1}^{N_{n}}\frac{1}{k!}%
||f_{k}-f_{k}^{n}||_{L^{2}[0,1]^{k}}^{2}+\sum_{k=N_{n}+1}^{\infty }\frac{1}{%
k!}||f_{k}||_{L^{2}[0,1]^{k}}^{2} \\
&<&\frac{1}{n^{2}}
\end{eqnarray*}
which implies that $\xi _{n}\rightarrow $ $\xi $ in $L^{2}(\mathbf{W}%
_{0}^{1},\mathcal{F}_{1},P^{w})$. We of course can choose $N_{n}$ increasing
in $n$. It is obvious that for each $k$, $f_{k}^{n}\rightarrow f_{k}$ as $%
n\rightarrow \infty $.

Let $Y(n)_{t}=E^{\mu }(\xi _{n}|\mathcal{F}_{t})$ and 
\begin{equation*}
Y(n)_{t}^{\varepsilon }=P_{-\log \sqrt{\varepsilon }}Y(n)_{t}=%
\sum_{k=1}^{N_{n}}\varepsilon ^{\frac{k}{2}}I_{k}(f_{k}^{n})_{t}
\end{equation*}

Let $\nu _{n}^{\varepsilon }$ denote the distribution of $%
(Y(n)_{t}^{\varepsilon })$.

Let $A_{n},B_{n},C_{j,n}$ be the corresponding vector fields determined in
Proposition \ref{prop0.1} for each $\xi _{n}$ in place of $\xi $. Then $%
Y(n)_{t}^{\varepsilon }=\pi _{1}(y_{t}^{\varepsilon ,n})$, where $%
(y_{t}^{\varepsilon ,n})_{t\geq 0}$ is the unique strong solution to 
\begin{eqnarray}
dy_{t} &=&\sqrt{\varepsilon }A_{n}(y_{t})\circ dw_{t}+\varepsilon
B_{n}(y_{t})dt  \notag \\
&&+\sum_{j=1}^{N_{n}-1}\varepsilon ^{\frac{j}{2}}C_{j,n}(t,y_{t})dt\text{ ,
\ \ }y_{0}=0  \label{n-1}
\end{eqnarray}
on $(\mathbf{W}_{0}^{d},\mathcal{F}_{1},P^{w})$. Let $X(n)_{t}^{\varepsilon
}=\pi _{1}(x_{t}^{\varepsilon ,n})$ where $(x_{t}^{\varepsilon ,n})_{t\geq
0} $ is the unique strong solution to 
\begin{equation}
dx_{t}=\sqrt{\varepsilon }A_{n}(x_{t})\circ dw_{t}\text{, \ }x_{0}=0\text{.}
\label{n-2}
\end{equation}

\begin{lemma}
\label{lem-c}Both families $\{\left( Y(n)_{t}^{\varepsilon }\right) _{t\leq
1}:\varepsilon \in (0,1)\}_{n=1,2,\cdots }$ and $\{\left(
X(n)_{t}^{\varepsilon }\right) _{t\leq 1}:\varepsilon \in
(0,1)\}_{n=1,2,\cdots }$ converge to $\{\left( Y_{t}^{\varepsilon }\right)
_{t\leq 1}:\varepsilon \in (0,1)\}$ exponentially. That is, for each $\delta
>0$, 
\begin{equation}
\lim_{N\rightarrow \infty }\lim_{\varepsilon \rightarrow 0}\varepsilon \log
P^{w}\left\{ \sup_{t\leq 1}\left| Y(n)_{t}^{\varepsilon }-Y_{t}^{\varepsilon
}\right| \geq \delta \right\} =-\infty  \label{ex-01}
\end{equation}
and 
\begin{equation}
\lim_{N\rightarrow \infty }\lim_{\varepsilon \rightarrow 0}\varepsilon \log
P^{w}\left\{ \sup_{t\leq 1}\left| X(n)_{t}^{\varepsilon }-Y_{t}^{\varepsilon
}\right| \geq \delta \right\} =-\infty \text{.}  \label{ex-02}
\end{equation}
\end{lemma}

\begin{proof}
By Proposition \ref{lem-b1}, we only need to show (\ref{ex-01}). By Lemma 
\ref{lemn01}, for any $\delta >0$, $\varepsilon \in (0,1)$ we have 
\begin{equation*}
P^{w}\left\{ \sup_{t\leq 1}\left| Y(n)_{t}^{\varepsilon }-Y_{t}^{\varepsilon
}\right| \geq \delta \right\} \leq (1+\varepsilon )^{1+\frac{1}{\varepsilon }%
}\frac{||\xi _{n}-\xi ||_{2}^{1+\frac{1}{\varepsilon }}}{\delta ^{1+\frac{1}{%
\varepsilon }}}\text{ }
\end{equation*}
so that 
\begin{eqnarray*}
&&\varepsilon \log P^{w}\left\{ \sup_{t\leq 1}\left| Y(n)_{t}^{\varepsilon
}-Y_{t}^{\varepsilon }\right| \geq \delta \right\}  \\
&\leq &\varepsilon \left( 1+\frac{1}{\varepsilon }\right) \log
(1+\varepsilon )-\varepsilon \left( 1+\frac{1}{\varepsilon }\right) \log
\delta  \\
&&+\varepsilon \left( 1+\frac{1}{\varepsilon }\right) \log ||\xi _{n}-\xi
||_{2}\text{ . }
\end{eqnarray*}
Hence 
\begin{eqnarray*}
&&\lim_{\varepsilon \rightarrow 0}\varepsilon \log P^{w}\left\{ \sup_{t\leq
1}\left| Y(n)_{t}^{\varepsilon }-Y_{t}^{\varepsilon }\right| \geq \delta
\right\}  \\
&\leq &\log ||\xi _{n}-\xi ||_{2}-\log \delta 
\end{eqnarray*}
and therefore 
\begin{equation*}
\lim_{N\rightarrow \infty }\lim_{\varepsilon \rightarrow 0}\varepsilon \log
P^{w}\left\{ \sup_{t\leq 1}\left| Y(n)_{t}^{\varepsilon }-Y_{t}^{\varepsilon
}\right| \geq \delta \right\} =-\infty \text{.}
\end{equation*}
\end{proof}

To prove the large deviation principle for the limit distributions of $%
(Y_{t}^{\varepsilon })_{t\leq 1}$, one would attempt to apply an extended
contraction principle (for example, Theorem 4.2.23, page 133, \cite{LD
techniques by Dembo and Zeitouni}) to the exponential approximations $%
X(n)^{\varepsilon }$. Since $\tilde{F}_{n}\left( \mathbf{w}\right) _{t}=\pi
_{1}(\tilde{G}_{n}(\Gamma (\varepsilon )\mathbf{w})_{0,t}^{1})$ is a version
of $X(n)_{t}^{\varepsilon }$, which approximate $\{Y^{\varepsilon
}:\varepsilon \in (0,1)\}$ exponentially, where $\tilde{G}_{n}$ is the
It\^{o}-Lyons mapping on $(\mathbb{W}^{p},d_{p})$ associated with 
\begin{equation}
dx_{t}=A_{n}(x_{t})\circ dw_{t}\text{, \ }x_{0}=0\text{.}  \label{n-3}
\end{equation}
The mapping $F_{n}:(\mathbb{W}^{p},d_{p})\rightarrow (\mathbf{W}%
_{0}^{d},||\cdot ||)$ is continuous, and the distribution family of $%
\{(X(n)_{t}^{\varepsilon })_{t\leq 1}:\varepsilon \in (0,1)\}$ satisfies the
large deviation principle with rate function given by 
\begin{equation*}
I_{n}^{\prime }(w)=\inf \left\{ I(h)\mid h\in H\text{ such that }F_{n}(%
\mathbf{h})=w\right\}
\end{equation*}
where 
\begin{equation*}
\tilde{F}_{n}(\mathbf{h})_{t}=\sum_{k=1}^{N_{n}}\int_{0}^{t}f_{k}^{n}(t_{1},%
\cdots ,t_{k})\dot{h}(t_{1})\cdots \dot{h}(t_{k})dt_{1}\cdots dt_{k}\text{ \
\ \ }\forall h\in H\text{.}
\end{equation*}
Therefore, if $\tilde{F}_{n}$ were convergent uniformly (in $p$-variation
distance) on any level set $\{I(h)\leq L\}$ uniformly, one could conclude
the proof of Theorem \ref{LD for Martingale}.

However, unfortunately, as a matter of fact, $\tilde{F}_{n}$ does not
converge uniformly on $\{I(h)\leq L\}$ in $p$-variation metric $d_{p}$ in
general, which would require a control on the derivatives of $f_{k}$. Thus,
we can not prove our main theorem by simply appealing to a (generalized)
contraction principle.

This is the reason why we develop the continuity theorem for large
deviations, Theorem \ref{limit-LDP}.

\paragraph{Proof of Theorem \ref{LD for Martingale}}

Let $E=\mathbf{W}_{0}^{d}$ with the uniform norm, $H=H_{0}^{1}([0,1];R^{d})$
with the Sobolev norm $||\cdot ||_{H^{1}}$.\ Then $I$ is a good rate
function with the effective set $\{I<\infty \}=H$. For each $N$ consider the
following mappings $F_{n}:H\rightarrow E$, where 
\begin{equation*}
F_{n}(h)_{t}=\sum_{k=1}^{N_{n}}\int_{0<t_{1}<\cdots
<t_{k}<t}f_{k}^{n}(t_{1},\cdots ,t_{k})\dot{h}(t_{1})\cdots \dot{h}%
(t_{k})dt_{1}\cdots dt_{k}\text{ \ \ \ }t\in \lbrack 0,1]
\end{equation*}
and $F:H\rightarrow E$ by 
\begin{equation*}
F(h)_{t}=\sum_{k=1}^{\infty }\int_{0<t_{1}<\cdots
<t_{k}<t}f_{k}(t_{1},\cdots ,t_{k})\dot{h}(t_{1})\cdots \dot{h}%
(t_{k})dt_{1}\cdots dt_{k}\text{\ \ \ \ \ }t\in \lbrack 0,1]\text{ .}
\end{equation*}
Then, according to Proposition \ref{prop-r3}, $F_{n}$, $F$ are rate-function
mappings. It is easy to see that $F_{n}\rightarrow F$ uniformly on any level
set $K_{L}\equiv \{h:I(h)\leq L\}$ with respect to the uniform norm, thus,
according to Corollary \ref{coro0z}, for each $n$, both the laws of $%
\{X(n)^{\varepsilon }:\varepsilon \in (0,1)\}$ and $\{Y(n)^{\varepsilon
}:\varepsilon \in (0,1)\}$ satisfy the large deviation principle on $(%
\mathbf{W}_{0}^{d},||\cdot ||)$ with the common rate function 
\begin{equation*}
I_{n}^{\prime }(s)=\inf \left\{ I(h):h\in H\text{ \ \ s.t. \ }%
F_{n}(h)=s\right\}
\end{equation*}
and $\{Y(n)^{\varepsilon }:\varepsilon \in (0,1)\}$ goes to $%
\{Y^{\varepsilon }:\varepsilon \in (0,1)\}$ exponentially, therefore, by
Theorem \ref{limit-LDP}, $\{Y^{\varepsilon }:\varepsilon \in (0,1)\}$
satisfies the large deviation principle with rate function 
\begin{equation*}
I^{\prime }(s)=\inf \left\{ I(h):h\in H\text{ \ \ s.t. \ }F(h)=s\right\} 
\text{.}
\end{equation*}

\vskip0.5truecm

\noindent {\textbf{Acknowledgments.}} The first author would like to thank
Professor M. Ledoux for his comments, and to thank Professor Quansheng Liu
for his references on large deviations for martingales in discrete-time. The
research of the paper was partly supported by EPSRC grant EP/F029578/1.

\bigskip

\noindent{\small Z. Qian and C. Xu, Mathematical Institute, University of
Oxford, 24 - 29 St. Giles', Oxford OX1 3LB}

\vskip0.3truecm

\noindent {\small Email: \texttt{qianz@maths.ox.ac.uk}}

\end{document}